\documentclass[11pt]{article}
\usepackage[width=15.5cm,height=21cm]{geometry}
\usepackage[colorlinks,allcolors=blue]{hyperref}

\usepackage{cite}
\usepackage{comment}

\newcommand{\doi}[1]{\href{https://doi.org/#1}{\url{doi:#1}}}
\newcommand{\myurl}[1]{\href{#1}{\url{#1}}}

\usepackage[intlimits]{amsmath}
\usepackage{amsfonts,amssymb,amsthm}
\usepackage{extarrows}

\newtheorem{thm}{Theorem}[section]
\newtheorem{cor}[thm]{Corollary}
\newtheorem{lem}[thm]{Lemma}
\newtheorem{prop}[thm]{Proposition}

\theoremstyle{definition}
\newtheorem{example}[thm]{Example}
\newtheorem{test}[thm]{Test}

\newtheorem{rem}[thm]{Remark}

\newcommand{\bC}{\mathbb{C}}
\newcommand{\bR}{\mathbb{R}}
\newcommand{\bN}{\mathbb{N}}
\newcommand{\bNz}{\mathbb{N}_0}
\newcommand{\bone}{\boldsymbol{1}}

\newcommand{\cB}{\mathcal{B}}
\newcommand{\cC}{\mathcal{C}}
\newcommand{\cG}{\mathcal{G}}
\newcommand{\cF}{\mathcal{F}}
\newcommand{\cH}{\mathcal{H}}
\newcommand{\cM}{\mathcal{M}}
\newcommand{\cP}{\mathcal{P}}
\newcommand{\cS}{\mathcal{S}}
\newcommand{\ba}{\mathbf{a}}
\newcommand{\al}{\alpha}
\newcommand{\be}{\beta}
\newcommand{\ga}{\gamma}
\newcommand{\de}{\delta}
\newcommand{\La}{\Lambda}

\newcommand{\si}{\sigma}
\renewcommand{\phi}{\varphi}
\newcommand{\tmu}{\widetilde{\mu}}
\newcommand{\enumber}{\operatorname{e}}
\newcommand{\imunit}{\operatorname{i}}

\renewcommand{\Im}{\operatorname{Im}}
\newcommand{\dif}{\mathrm{d}}

\newcommand{\conjw}{\overline{w}}
\newcommand{\conj}[1]{\overline{#1}}
\def\multiset#1#2{\ensuremath{\left(\kern-.3em\left(\genfrac{}{}{0pt}{}{#1}{#2}\right)\kern-.3em\right)}}

\usepackage{mathtools}
\newcommand{\eqdef}{\coloneqq}

\usepackage{tikz}
\usetikzlibrary{arrows, arrows.meta}
\definecolor{darkgreen}{rgb}{0, 0.5, 0}
\definecolor{lightgreen}{rgb}{0.7, 1, 0.7}

\title{Horizontal Fourier transform\\ of the polyanalytic Fock kernel}
\author{Erick Lee-Guzm\'an,
Egor A. Maximenko,\\
Gerardo Ramos-Vazquez,
Armando S\'anchez-Nungaray}

\begin{document}
\maketitle

\begin{abstract}
Let $n,m\ge 1$ and $\alpha>0$.
We denote by $\mathcal{F}_{\alpha,m}$ the $m$-analytic Bargmann--Segal--Fock space, i.e.,
the Hilbert space of all $m$-analytic functions defined on $\mathbb{C}^n$ and square integrables with respect to the Gaussian weight
$\exp(-\alpha |z|^2)$.
We study the von Neumann algebra $\mathcal{A}$ of bounded linear operators acting in $\mathcal{F}_{\alpha,m}$ and commuting with all ``horizontal'' Weyl translations,
i.e., Weyl unitary operators associated to the elements of $\mathbb{R}^n$.
The reproducing kernel of $\mathcal{F}_{1,m}$ was computed by Youssfi
[Polyanalytic reproducing kernels in $\mathbb{C}^n$,
Complex Anal. Synerg.,
2021, 7, 28].
Multiplying the elements of $\mathcal{F}_{\alpha,m}$ by an appropriate weight, we transform this space into another reproducing kernel Hilbert space whose kernel $K$ is invariant under horizontal translations.
Using the well-known Fourier connection between Laguerre and Hermite functions,
we compute the Fourier transform of $K$ in the ``horizontal direction'' 
and decompose it into the sum of $d$ products of Hermite functions,
with $d=\binom{n+m-1}{n}$.
Finally, applying the scheme proposed by Herrera-Ya\~{n}ez, Maximenko, Ramos-Vazquez
[Translation-invariant operators in reproducing kernel Hilbert spaces, Integr. Equ. Oper. Theory, 2022, 94, 31],
we show that $\mathcal{F}_{\alpha,m}$ is 
isometrically isomorphic to the space of vector-functions $L^2(\mathbb{R}^n)^d$,
and $\mathcal{A}$ is isometrically isomorphic to the algebra of matrix-functions
$L^\infty(\mathbb{R}^n)^{d\times d}$.


\end{abstract}

\bigskip
\subsection*{Funding}

The second author has been partially supported by Proyecto CONAHCYT ``Ciencia de Frontera''
FORDECYT-PRONACES/61517/2020
and by IPN-SIP projects (Instituto Polit\'{e}cnico Nacional, Mexico).\\
The third author has been supported by postdoctoral grant
(CONAHCYT, Mexico).\\
The fourth author has been supported by CONAHCYT grant 280732.

\clearpage

\tableofcontents

\bigskip

\section{Introduction}

\subsection*{Background}

The polyanalytic Bargmann--Segal--Fock spaces 
(shortly, polyanalytic Fock spaces)
are important for some problems of quantum mechanics~\cite{AbreuFeichtinger,Balk1991,KellerLuef2021,MouaynMoize2021,Shigekawa1987} and have connection with signal processing~\cite{Abreu2009,Abreu2012,HutnikHutnikova2014}.

Vasilevski~\cite{Vasilevski2000polyFock} studied the structure of the polyanalytic Fock space,
which we denote by
$\cF_{1,m}(\bC)$,
by ``cancelling'' the weight (see Remark~\ref{rem:Vasilevski_flatten_Fock})
and applying the Fourier transform in the imaginary direction.
Thereby, he constructed an isometric isomorphism between $\cF_{1,m}(\bC)$
and $L^2(\bR)^m$.
S\'anchez-Nungaray, Gonz\'alez-Flores,  L\'opez-Mart\'inez,
and Arroyo-Neri~\cite{SanchezGonzalezLopezArroyo2018} developed this idea and showed that the Toeplitz operators, acting in $\cF_{1,m}(\bC)$ and generated by bounded horizontal symbols,
can be converted into matrix families.

In the $n$-dimensional case,
it is possible to define $\be$-analytic functions for each multiindex $\be\in\bN^n$;
see, e.g., Balk~\cite[Section~6.4]{Balk1991} and Vasilevski~\cite{Vasilevski2023}.
In this paper, we prefer to work in spaces of ``homogeneously $m$-analytic functions'',
where $m\in\bN$.
These spaces are studied
by Leal-Pacheco,
Maximenko,
Ramos-Vazquez~\cite{LealMaximenkoRamos2021} and
Youssfi~\cite{Youssfi2021}.
In particular, we denote by $\cF_{\al,m}(\bC^n)$ or shortly by $\cF_{\al,m}$ the $m$-analytic Fock space,
see details in Section~\ref{sec:poly_Fock_space}.

Arroyo Neri,
S\'anchez-Nungaray,
Hern\'andez Marroquin,
and
L\'opez-Mart\'inez
\cite{ArroyoSanchezHernandezLopez2021}
generalized ideas from~\cite{SanchezGonzalezLopezArroyo2018} to the $n$-dimensional case.
They worked in polyanalytic and true-polyanalytic Fock spaces, including a space denoted by $F_{k}^2(\bC^n)$ in~\cite{SanchezGonzalezLopezArroyo2018}, which was defined as a certain direct sum of true-poly-Fock spaces.
As we show in Section~\ref{sec:true_poly_Fock},
$F_k^2(\bC^n)$
coincides with 
$\cF_{\al,m}(\bC^n)$
for $\al=1$ and $m=k+1-n$.
The authors of~\cite{SanchezGonzalezLopezArroyo2018} adapted for several variables the scheme introduced by Vasilevski~\cite{Vasilevski2000polyFock} and concluded that $\cF_{k}^{2}(\bC^n)$
is isometrically isomorphic to
$L^2(\bR^n)^d$,
with $d=\binom{k}{k-n}$.
Using that isomorphism,
they characterized the $C^*$-algebra generated by Toeplitz operators with extended horizontal symbols as the 
$C^*$-algebra of matrix-valued functions with limit in each ray to infinity.

In this paper,
we employ a different aproach.
We pass from $\cF_{\al,m}$ to its  ``flattened'' version $\cH_m$,
which is embedded into $L^2(\bR^{2n})$
and is invariant under the horizontal translations.
We compute the horizontal Fourier transform of the reproducing kernel of $\cH_m$ and describe the W*-algebra of operators invariant under the horizontal Weyl operators in $\cF_{\al,m}$ or, equivalently, the W*-algebra of the operators invariant under the horizontal translations in $\cH_m$.

\subsection*{Main result}

Given a Hilbert space $H$ and an autoadjoint subset $X$ of $\cB(H)$,
we denote by $\cC(X)$ the centralizer (commutant) of $X$,
which is the set of all operators in $\cB(H)$ that commute with every operator in $X$.
It is well known that $\cC(X)$ is a von Neumann
subalgebra of $\cB(H)$.
Given a unitary representation $(\rho,H)$ of a group $G$,
we denote by $\cC(\rho)$ the centralizer of the set $\{\rho(a)\colon\ a\in G\}$.

In this paper, we consider the unitary representation of $\bR^n$ in the space $\cF_{\al,m}$, given by the ``horizontal Weyl operators''
\[
(\rho_{\cF_{\al,m}}(a) f)(z)
\eqdef
f(z-a)\,
\enumber^{\al\langle z,a\rangle-\frac{\al}{2}|a|^2}
\qquad(a\in\bR^n,\ f\in\cF_{\al,m},\ z\in\bC^n).
\]
Our main result is an explicit decomposition of
$\cC(\rho_{\cF_{\al,m}})$:
\begin{equation}
\label{eq:isomorphic_algebras}
\cC(\rho_{\cF_{\al,m}})
\cong
\int_{\bR^n}^{\oplus} \cM_d\,\dif{}\mu_n(\xi)
= L^\infty(\bR^n)\otimes\cM_d
= L^\infty(\bR^n,\cM_d)
= L^\infty(\bR^n)^{d\times d}.
\end{equation}
Here $\cM_d$ is the algebra of $d\times d$ complex matrices.

\subsection*{Structure of the paper}

In Section~\ref{sec:Laguerre_summation_formula}, we state a formula for a sum of products of Laguerre polynomials.
In Section~\ref{sec:poly_Fock_space}, we recall the formula found by Youssfi~\cite{Youssfi2021} for the reproducing kernel of $\cF_{\al,m}$ with $\al=1$ and
pass from $\al=1$ to the general $\al>0$.
Using the formula from Section~\ref{sec:Laguerre_summation_formula}, we decompose this reproducing kernel to a certain sum of products, where each factor depends only on one coordinate.
In Section~\ref{sec:true_poly_Fock}, we show that this decomposition of the kernel is equivalent to the descomposition of $\cF_{\al,m}$ previously found by 
Arroyo Neri,
Hern\'{a}ndez Marroquin,
L\'{o}pez-Mart\'{i}nez,
and S\'{a}nchez-Nungaray
in~\cite{ArroyoSanchezHernandezLopez2021}.
In Section~\ref{sec:flattened_poly_Fock}, we construct an isometric isomorphism
$U_{\cF_{\al,m}}^{\cH_m}$
which multiplies the functions from $\cF_{\al,m}$ by a certain weight.
The resulting Hilbert space $\cH_m$ is an RKHS embedded into
$L^2(\bR^{2n},\tmu_{2n})$,
where $\tmu_{n}=(2\pi)^{-n/2}\mu_{n}$.
In Section~\ref{sec:Weyl},
we introduce horizontal Weyl translations in $\cF_{\al,m}$ and show that $U_{\cF_{\al,m}}^{\cH_m}$ intertwines them with unweighted horizontal translations in $\cH_m$.

In Section~\ref{sec:Fourier_connection_Laguerre_Hermite}, we recall a well-known formula for the Fourier transform of the function $u\mapsto\ell_j(u^2+a^2)$, where $\ell_j$ is the $j$th Laguerre function.
In Section~\ref{sec:Fourier_transform_of_reproducing_kernel}, using the formula from Section~\ref{sec:Fourier_connection_Laguerre_Hermite},
we compute the ``horizontal Fourier transform'' $F\otimes I$ of the reproducing kernel $K^{\cH_m}_{0,y}$,
and decompose the result into a sum of products:
\begin{equation}
\label{eq:F_times_I_K_eq_sum_q_q}
((F\otimes I)K^{\cH_m}_{0,y})(\xi,v)
=\sum_{\substack{k\in\bNz^n\\|k|\le m-1}}
q_{k,\xi}(v)
\overline{q_{k,\xi}(y)}.
\end{equation}
Using~\eqref{eq:F_times_I_K_eq_sum_q_q} and the scheme from~\cite{HerreraMaximenkoRamos2022},
in Section~\ref{sec:decomposition_of_spaces} we construct an isometric isomorphism $R_{\cH_m}$ between $\cH_m$ and $L^2(\bR)^d$,
where
\[
d = \#\{k\in\bNz^n\colon\ |k|\le m-1\}
= \binom{n+m-1}{n}.
\]
According to the scheme from~\cite{HerreraMaximenkoRamos2022},
$R_{\cH_m}$ intertwines translation-invariant operators in $\cH_m$ and multiplication operators (by matrix-functions) in $L^2(\bR,\tmu_n)^d$.
Thereby, we obtain~\eqref{eq:isomorphic_algebras}.
Finally, in Section~\ref{sec:tests} we describe numerical tests that verify many formulas from this paper for small values of $n$ and $m$.

\subsection*{Novelty}

This paper is inspired by~\cite{Vasilevski2000polyFock,SanchezGonzalezLopezArroyo2018,ArroyoSanchezHernandezLopez2021}
and contains similar ideas.
Let us emphasize some new aspects of this paper.
\begin{itemize}
\item We consider not only Toeplitz operators with symbols invariant under horizontal translations, but all bounded operators commuting with $\rho_{\cF_{m,\al}}(a)$ for all $a$ in $\bR^n$.

\item Instead of decomposing the whole poly-Fock space into ``true'' poly-Fock subspaces, 
we decompose the reproducing kernel into a sum of products.

\item Instead of applying the Fourier transform to differential equations, we apply it to the reproducing kernel.

\item Working with reproducing kernels allows us to construct several new examples and verify many formulas of this paper by calculations in SageMath.

\end{itemize}

\section{Summation formulas for Laguerre polynomials}
\label{sec:Laguerre_summation_formula}

We employ the usual notation $L_n^{(\al)}$ for the generalized Laguerre(--Sonin) polynomials.
The following formula for their generating function is well known
(see, e.g., \cite[formula (5.1.9)]{Szego1975}):
\begin{equation}
\label{eq:Laguerre_generating_function}
\sum_{p=0}^\infty
L_p^{(\al)}(x)\,t^p
=
\frac{1}{(1-t)^{\al+1}}
\exp\left(-\frac{tx}{1-t}\right).
\end{equation}

\begin{prop}
\label{prop:Laguerre_of_sum}
For every $\al,\be>-1$,
every $p$ in $\bNz$,
and every $x,y$ in $\bC$,
\begin{equation}
\label{eq:Laguerre_of_sum}
L_p^{(\al+\be+1)}(x+y) 
= \sum_{k=0}^p
L_k^{(\al)}(x)\,
L_{p-k}^{(\be)}(y).
\end{equation}
\end{prop}

\begin{proof}
This fact is mentioned without proof in~\cite[Section~10.12, formula~(41)]{Erdelyi1953} and
\cite[???]{Howlett1966}.
It follows easily from~\eqref{eq:Laguerre_generating_function}:
\begin{align*}
\sum_{p=0}^\infty
L_p^{(\al+\be+1)}(x+y)\,t^p
&=
\frac{1}{(1-t)^{\al+\be+2}}
\exp\left(-\frac{t(x+y)}{1-t}\right)
\\
&=
\frac{1}{(1-t)^{\al+1}}
\exp\left(-\frac{tx}{1-t}\right)
\frac{1}{(1-t)^{\be+1}}
\exp\left(-\frac{ty}{1-t}\right)
\\
&=
\left(
\sum_{k=0}^\infty
L_k^{(\al)}(x)\,t^k
\right)
\left(
\sum_{j=0}^\infty
L_j^{(\be)}(y)\,t^j
\right).
\end{align*}
Equating the coefficients of $t^p$ we obtain~\eqref{eq:Laguerre_of_sum}.
\end{proof}

For $\be=0$, \eqref{eq:Laguerre_of_sum} reduces to
\begin{equation}
\label{eq:sum_laguerre}
\sum_{k=0}^p L_k^{(\al)}(x)
=L_p^{(\al+1)}(x).
\end{equation}
The main result of this section, Proposition~\ref{prop:Laguerre_decomposition},
will be a generalization of Proposition~\ref{prop:Laguerre_of_sum}.

In what follows,
given $n,m$ in $\bN$,
we denote by $J_{n,m}$ the following set of multiindices:
\begin{equation}
\label{eq:J_def}
J_{n,m}
\eqdef
\bigl\{k\in\bNz^n\colon\ |k|\le m-1\bigr\}.
\end{equation}
Obviously, there is a bijection from $J_{n,m}$ onto 
\[
\{(k_1,\ldots,k_{n+1})\in\bNz^{n+1}\colon\ k_1+\ldots+k_{n+1}=m-1\}.
\]
Hence, the size of $J_{n,m}$ can be expressed through a certain multiset coefficient.
We denote the size of $J_{n,m}$ by $d_{n,m}$:
\begin{equation}
\label{eq:J_cardinality}
d_{n,m}
\eqdef \#J_{n,m}
=\multiset{m-1}{n+1}
=\binom{n+m-1}{m-1}
=\binom{n+m-1}{n}
=\frac{(n+m-1)!}{n!\,(m-1)!}.
\end{equation}
Another way to prove \eqref{eq:J_cardinality} is to represent $J_{n,m}$ as
$J_{n,m} = \bigcup_{j=0}^{m-1} \{k\in\bNz^n\colon |k|=j\}$.

\begin{prop}
\label{prop:Laguerre_decomposition}
For every $n$ in $\bN$,
every $p$ in $\bNz$,
and every $t=(t_1,\ldots,t_n)$ in $\bC^n$,
\begin{equation}\label{eq:Laguerre_decomposition}
L_p^{(n)}(t_1 + \cdots + t_n)
=\sum_{k\in J_{n,p+1}}\;
\prod_{r=1}^n L_{k_r}(t_r).
\end{equation}
The number of summands in the sum is $\binom{n+p}{n}$.
\end{prop}

\begin{proof}
We will proceed by mathematical induction over $n$.
For every $n$ in $\bN$,
we denote by $\cP(n)$ the statement that~\eqref{eq:Laguerre_decomposition} holds for all $p$ in $\bNz$ and all $t=(t_1,\ldots,t_n)$ in $\bC^n$.

For $n=1$, ~\eqref{eq:Laguerre_decomposition} converts to~\eqref{eq:sum_laguerre} with $\al=0$.
Therefore, $\cP(1)$ is true.

Suppose that $n$ in $\bN$ and $\cP(n)$ holds. Let us prove that $\cP(n+1)$ holds.
Let $d\in\bNz$, $t=(t_1,\ldots,t_{n+1})\in\bC^{n+1}$,
and $u\eqdef t_1+\ldots+t_n+t_{n+1}$.
Apply~\eqref{eq:Laguerre_of_sum}
with $\al=n$, $\be=0$,
$x=t_1+\cdots+t_n$,
and $y=t_{n+1}$:
\[
L_d^{(n+1)}(u)
=\sum_{j=0}^d
L_{d-j}^{(n)}(t_1+\cdots+t_n)
L_j(t_{n+1}).
\]
In the term with index $j$,
we apply the induction hypothesis $\cP(n)$ with parameters $d-j$ and $(t_1,\ldots,t_n)$.
After that, we denote $j$ by $k_{n+1}$.
\[
L_d^{(n+1)}(u)
=\sum_{j=0}^d
\Biggl(\,\sum_{\substack{k\in\bNz^n\\|k|\le d-j}}
\prod_{r=1}^n
L_{k_r}(t_r)
\Biggr)
L_j(t_{n+1})
=\sum_{\substack{(k_1,\ldots,k_{n+1})\in\bNz^{n+1}\\k_1+\cdots+k_n+k_{n+1}\le d}}
\prod_{r=1}^{n+1}
L_{k_r}(t_r).
\]
Thus, $\cP(n+1)$ is proven.
\end{proof}

\section{Polyanalytic Bargmann--Segal--Fock space}
\label{sec:poly_Fock_space}

Let $\bN\eqdef\{1,2,\ldots\}$, $\bNz\eqdef\{0\}\cup\bN=\{0,1,2,\ldots\}$.
Let $n,m\in\bN\eqdef\{1,2,\ldots\}$.
For a multiindex $k$ in $\bNz^n$, we use the standard notation
\[
|k|\eqdef\sum_{r=1}^n k_r,\qquad
k!\eqdef\prod_{r=1}^n k_r!.
\]
Following~\cite{HerreraMaximenkoRamos2022},
we say that a smooth function $f\colon\bC^n\to\bC$ is $m$-analytic (or homogeneously $m$-analytic),
if $\overline{D}^k f=0$ for every multiindex $k$ with $|k|=m$.
Here $\overline{D}^k$ is the usual Wirtinger operator. It was proven in~\cite{LealMaximenkoRamos2021} that every $m$-analytic function admits a decomposition of the form
\begin{equation}
\label{eq:polyanalytic_functions_decomposition}
f(z) = \sum_{\substack{k\in\bNz^n\\|k|\le m-1}} a_k(z)\conj{z}^k,
\end{equation}
where $a_k\colon\bC^n\to\bC$ are some analytic functions.

We denote by $\langle\cdot,\cdot\rangle$ the inner product in $\bC^n$ and by $|\cdot|$ the euclidian norm in $\bC^n$:
\[
\langle z,w\rangle
\eqdef\sum_{r=1}^n z_r\,\overline{w_r},\qquad
|z|=\sqrt{\langle z,z\rangle}.
\]
Let $m\in\bN$, $\al>0$.
We denote by $\cF_{\al,m}(\bC^n)$ (or simply  by $\cF_{\al,m}$) the space of all $m$-analytic functions $\bC^n\to\bC$, square integrable with respect to the Gaussian weight $\frac{\al^n}{\pi^n}e^{-\alpha|z|^2}$. The norm on $\cF_{\al,m}$ is defined by
\[
\|f\|_{\cF_{\al,m}}
\eqdef
\left(\frac{\al^n}{\pi^n}
\int_{\bC^n} |f(z)|^2
\enumber^{-\al|z|^2}\,\dif\mu_{2n}(z)\right)^{1/2},
\]
where $\mu_{2n}$ is the Lebesgue measure on $\bC^n$.
This space is known as the ($n$-dimensional) $m$-analytic Bargmann--Segal--Fock space.

Askour, Intisar, and Mouyayn~\cite{AskourIntissarMouayn1997}
computed the reproducing kernel of this space for $n=1$ and $\al=1$:
\begin{equation}
\label{eq:K_Fm_onedimensional}
K_z^{\cF_{1,m}(\bC)}(w)
=\enumber^{\overline{z}w}L_{m-1}^{(1)}\left(|w-z|^2\right).
\end{equation}
Recently, Youssfi~\cite{Youssfi2021} computed the reproducing kernel of $\cF_{1,m}(\bC^n)$:
\begin{equation}
\label{eq:K_Fm}
K_z^{\cF_{1,m}(\bC^n)}(w)
= \enumber^{\langle w,z\rangle}
L_{m-1}^{(n)}(|w-z|^2).
\end{equation}
In fact, the normalization of 
the Gaussian weight in~\cite{Youssfi2021} is different from the usual one that we accept in this paper.
So,~\eqref{eq:K_Fm} is a trivial adjustment of the formula that appears in~\cite{Youssfi2021}.

We define
$U_{\cF_{1,m}}^{\cF_{\al,m}}\colon
\cF_{1,m}\rightarrow\cF_{\al,m}$ by
\[
(U_{\cF_{1,m}}^{\cF_{\al,m}}f)(z)\eqdef f(\sqrt{\al}z).
\]

\begin{prop}
\label{prop:Fm_rightarrow_Fmalpha}
$U_{\cF_{1,m}}^{\cF_{\al,m}}$ is a well-defined isometric isomorphism of Hilbert spaces.
\end{prop}

\begin{proof}
We use the change of variable $w=\sqrt{\al}z$ in the integral:
\begin{align*}
\|U_{\cF_{1,m}}^{\cF_{\al,m}}f\|^2_{\cF_{\al,m}}
&=\frac{\al^n}{\pi^n}
\int_{\bC^n}\left|f(\sqrt{\al}z)\right|^2
\enumber^{-\al|z|^2}\,\dif\mu_{2n}(z)
=\frac{\al^n}{\pi^n}
\int_{\bC^n}|f(w)|^2
\enumber^{-|w|^2}\frac{\dif\mu_{2n}(w)}{\al^n}
\\
&=\frac{1}{\pi^n}
\int_{\bC^n}
|f(w)|^2
\enumber^{-|w|^2}\dif\mu_{2n}(z)
=\|f\|^2_{\cF_{1,m}}.
\qedhere
\end{align*}
\end{proof}

The following elementary result was proven in~\cite{LealMaximenkoRamos2021}; see also
\cite[Sections~5.6 and 5.7]{PaulsenRaghupathi2016}.

\begin{prop}
\label{prop:pushforward_RK}
Let $X,Y$ be non-empty sets,
$\psi\colon Y\to X$
and $J\colon Y\to\bC$
be some functions,
$H_1$ be a Hilbert space of functions over $X$
with reproducing kernel $(K_x^{H_1})_{x\in X}$,
$H_2$ be a Hilbert space of functions over $Y$, and
\[
(Uf)(z)\eqdef J(z)f(\psi(z))
\]
be a well-defined isometric isomorphism mapping $H_1$ onto $H_2$.
Then $H_2$ is an RKHS,
and its reproducing kernel
$(K^{H_2}_u)_{u\in Y}$ is given by
\[
K^{H_2}_u(v)
=\overline{J(u)}\,J(v)\,
K^{H_1}_{\psi(u)}(\psi(v)).
\]
\end{prop}

\begin{prop}
\label{prop:Fm_al_kernel}
$\cF_{\al,m}$ is an RKHS with the kernel 
\begin{equation}
\label{eq:K_Fm_al_kernel}
K_z^{\cF_{\al,m}}(w)
=\enumber^{\al \langle w,z \rangle}
L_{m-1}^{(n)}(\al|w-z|^2).
\end{equation}
The norm of this kernel is
\begin{equation}
\label{eq:K_Fm_al_norm}
\|K_z^{\cF_{\al,m}}\|
_{\cF_{\al,m}}
=\sqrt{d_{n,m}}\,
\enumber^{\frac{\al}{2}|z|^2}.
\end{equation}
\end{prop}

\begin{proof}
Apply Propositions~\ref{prop:Fm_rightarrow_Fmalpha} and \ref{prop:pushforward_RK}
with $J(z)=1$ and $\psi(z)=\sqrt{\al}z$:
\[K_u^{\cF_{\al,m}}(v)
=\overline{J(u)}J(v)K_{\psi(u)}^{\cF_{1,m}}(\psi(v))
=\enumber^{\al\langle v,u \rangle}
L_{m-1}^{(n)}(|\sqrt{\al}v-\sqrt{\al}u|^2)
=
\enumber^{\al\langle v,u \rangle}
L_{m-1}^{(n)}(\al|v-u|^2).
\]
To obtain~\eqref{eq:K_Fm_al_norm}
we evaluate the kernel on the diagonal
and recall that $d_{n,m}$ is defined by~\eqref{eq:J_cardinality}:
\[
\|K_z^{\cF_{\al,m}}\|
_{\cF_{\al,m}}
=\sqrt{K_z^{\cF_{\al,m}}(z)}
=\sqrt{L_{m-1}^{(n)}(0)
\enumber^{\al|z|^2}}
=\sqrt{\binom{n+m-1}{n}\,
\enumber^{\al|z|^2}}
=\sqrt{d_{n,m}}\,
\enumber^{\frac{\al}{2}|z|^2}.
\qedhere
\]
\end{proof}

Using the decomposition of generalized Laguerre polynomials given in Proposition~\ref{prop:Laguerre_decomposition}
we can decompose $K_z^{\cF_{\al,m}}(w)$ into a sum of products of $n$ factors where the $j$th factor involves only $z_j$ and $w_j$.

Recall that the Laguerre function $\ell_m$ is defined by
\begin{equation}
\label{eq:Laguerre_function_def}
\ell_m(t)
\eqdef \enumber^{-\frac{1}{2}t} L_m(t).
\end{equation}

\begin{cor}\label{cor:K_Fm_decomp}
For every $x,y,u,v$ in $\bR^n$,
\begin{equation}
\label{eq:K_Fm_via_sum_of_products_of_Laguerre_polynomials}
K_z^{\cF_{\al,m}}(w)
=\sum_{k\in J_{n,m}}\;
\prod_{r=1}^n \enumber^{\al w_r\conj{z_r}}
L_{k_r}(\al|w_r-z_r|^2).
\end{equation}
Equivalently,
\begin{equation}
\label{eq:K_Fm_via_sum_of_products_of_Laguerre_functions}
K_z^{\cF_{\al,m}}(w)=\sum_{k\in J_{n,m}}\;
\prod_{r=1}^n\enumber^{\frac{\alpha}{2}(|w_r|^2+|z_r|^2)+\imunit\alpha\Im(w_r\overline{z_r})}\ell_{k_r}(\alpha|w_r-z_r|^2).
\end{equation}
\end{cor}

\begin{proof}
Follows from Propositions~\ref{prop:Fm_al_kernel} and~\ref{prop:Laguerre_decomposition}.
\end{proof}

\section{Decomposition into ``true'' poly-Fock spaces}
\label{sec:true_poly_Fock}

In this section,
we explain connections of $\cF_{\al,m}(\bC^n)$ with spaces studied in~\cite{ArroyoSanchezHernandezLopez2021,AskourIntissarMouayn1997,Vasilevski2000polyFock}.

$\bigl(\cF_{\al,m}(\bC)\bigr)_{m\in\bN}$ is a strictly increasing sequence of closed subespaces of $L^2\left(\bC,\frac{\al}{\pi}\enumber^{-\al|z|^2}\mu_{2}\right)$.
For every $m$ in $\bN$, Vasilevski~\cite{Vasilevski2000polyFock} defined the true-$m$-Fock space
as the orthogonal complement of $\cF_{\al,m-1}(\bC)$ inside $\cF_{\al,m}(\bC)$, i.e.,
\[
\cF_{\al,(m)}(\bC) \eqdef \cF_{\al,m}(\bC) \ominus \cF_{\al,m-1}(\bC).
\]
We set $\cF_{\al,(0)}(\bC) \eqdef \cF_{\al,0}(\bC)\eqdef\{0\}$.
Askour, Intissar, and Mouayn~\cite{AskourIntissarMouayn1997} proved that $\cF_{\al,(m)}(\bC)$ is an RKHS whose reproducing kernel is
\[
K^{\cF_{\al,(m)}(\bC)}_z(w) = \enumber^{\al w\conj{z}}
L_{m-1}(\al|w-z|^2).
\]
For each multiindex $\be$ in $\bN^n$,
Arroyo Neri, S\'{a}nchez-Nungaray, Hern\'{a}ndez Marroquin, and L\'{o}pez-Mart\'{i}nez \cite{ArroyoSanchezHernandezLopez2021} introduced the true poly-Fock space over $\bC^n$ as
\begin{equation}\label{def:true_polyFock_multiindex}
\cF_{\al,(\beta)}(\bC^n)\eqdef\cF_{\al,(\beta_1)}(\bC)
\otimes \cdots \otimes \cF_{\al,(\beta_n)}(\bC).
\end{equation}
See also~\cite[Section 4]{Vasilevski2023} for a generalization with a vectorial parameter.
According to the general theory of RKHS~\cite[Theorem 5.11]{PaulsenRaghupathi2016}, this tensor product is also a RKHS, and its reproducing kernel is
\begin{equation}\label{eq:K_true_polyFock_multiindex}
K^{\cF_{\al,(\be)}(\bC^n)}_z(w) 
= \prod_{r=1}^n \enumber^{\al \langle w,z\rangle}
L_{\be_{r-1}}(\al|w_r-z_r|^2).
\end{equation}
Corollary~\ref{cor:K_Fm_decomp} implies that for every $m$ in $\bN$,
\begin{equation}\label{eq:true_polyFock_connection}
\cF_{\al,m}(\bC^n) = \bigoplus_{\substack{\beta\in\bN^n\\ n\le|\be|\le m+n-1}} \cF_{\al,(\be)}(\bC^n)
= \bigoplus_{k\in J_{n,m}} \cF_{\al,(k+\bone_n)}(\bC^n),
\end{equation}
where $\bone_n=(1, \ldots, 1)\in\bNz^n$.
This decomposition can also be proved by using unitary operators from~\cite{Vasilevski2000polyFock,ArroyoSanchezHernandezLopez2021}, which are similar to the operator $R_{\cF_{\al,m}}$ constructed in Section~\ref{sec:decomposition_of_spaces}.

\section{Flattened polyanalytic Bargmann--Segal--Fock space}
\label{sec:flattened_poly_Fock}

We denote by $\tmu_n$ the Lebesgue measure $\mu_n$ multiplied by $(2\pi)^{-n/2}$.
For each $f$ in $\cF_{\al,m}(\bC^n)$,
we denote by $U_{\cF_{\al,m}}^{\cH_m}f$ the function defined on $\bR^{2n}$ by
\begin{equation}
\label{def:U_Fm_Hm}
(U_{\cF_{\al,m}}^{\cH_m}f)(x,y)
\eqdef 
2^{\frac{n}{2}}
\enumber^{-\frac{1}{2}|x|^2
-\frac{1}{2}|y|^2
-\imunit\,\langle x,y\rangle}
f\left(\frac{x+\imunit y}{\sqrt{\al}}\right)
\qquad(x,y\in\bR^n).
\end{equation}
Let $\cH_m$ be the space 
$\{U_{\cF_{\al,m}}^{\cH_m}f\colon\ f\in\cF_{\al,m}\}$, provided with the inner product and norm from
$L^2(\bR^{2n},\tmu_{2n})$.
After proving formula~\eqref{eq:K_Hm} below we will see that $\cH_m$ does not depend on the parameter $\al$.

\begin{prop}
$\cH_m$ is a Hilbert space,
and $U_{\cF_{\al,m}}^{\cH_m}\colon\cF_{\al,m}\to\cH_m$ is an isometric isomorphism.
The inverse isometric isomorphism
$U_{\cH_m}^{\cF_{\al,m}}$ acts by
\begin{equation}
\label{eq:U_Hm_to_Fm}
(U_{\cH_m}^{\cF_{\al,m}} g)(u+\imunit v)
=2^{-\frac{n}{2}}\enumber^{\frac{\al}{2}|u|^2+\frac{\al}{2}|v|^2+\imunit\al\langle u,v\rangle} g\left(\sqrt{\alpha}u,\sqrt{\al}v\right).
\end{equation}
\end{prop}

\begin{proof}
Let us verify the isometric property.
We apply the changes of variables $x=\sqrt{\al}u$, $y=\sqrt{\al}v$,
and then $z=u+\imunit v$:
\begin{align*}
\|U_{\cF_{\al,m}}^{\cH_m}f\|_{\cH_m}^2
&=
\frac{1}{(2\pi)^{n}}
\int_{\bR^{2n}}2^n
\enumber^{-(|x|^2+|y|^2)}
\left|
f\left(\frac{x+\imunit y}{\sqrt{\al}}\right)
\right|^2\,
\dif\mu_n(x)\,\dif\mu_n(y)
\\
&=
\frac{\al^n}{\pi^n}
\int_{\bR^{2n}}
\enumber^{-\al(|u|^2+|v|^2)}
\left|
f(u+\imunit v)
\right|^2\,
\dif\mu_n(u)\,\dif\mu_n(v)
\\
&=
\frac{\al^n}{\pi^n}
\int_{\bC^n} \enumber^{-\al|z|^2}\,|f(z)|^2\,\dif\mu_{2n}(z)
=
\|f\|_{\cF_{\al,m}}^2.
\end{align*}
By the definition of $\cH_m$,
$U_{\cF_{\al,m}}^{\cH_m}$ is surjective.
Finally, $\cH_m$ is complete because $\cF_{\al,m}$ is complete.
\end{proof}

\begin{prop}
\label{prop:K_Hm}
$\cH_m$ is a RKHS over $\bR^{2n}$.
Its reproducing kernel is
\begin{equation}
\label{eq:K_Hm}
K^{\cH_m}_{x,y}(u,v)
=
2^n\,
\enumber^{-\frac{1}{2}(|u-x|^2+|v-y|^2)
-\imunit\,\langle u-x,v+y\rangle}
L_{m-1}^{(n)}(|u-x|^2+|v-y|^2).
\end{equation}
The norm of this reproducing kernel is
\begin{equation}
\label{eq:K_Hm_norm}
\|K^{\cH_m}_{x,y}\|_{\cH_m}
=\sqrt{2^n d_{n,m}}.
\end{equation}
\end{prop}

\begin{proof}
By Proposition~\ref{prop:pushforward_RK}
and formula~\eqref{eq:K_Fm},
\begin{align*}
K^{\cH_m}_{x,y}(u,v)
&=
2^n
\enumber^{-\frac{1}{2}
\left(|u|^2+|v|^2
+2\imunit\,\langle u,v\rangle\right)}
\enumber^{-\frac{1}{2}
\left(|x|^2+|y|^2
-2\imunit\,\langle x,y\rangle\right)}
K^{\cF_{\al,m}}_{\frac{1}{\sqrt{\al}}\bigl(x+\imunit y\bigr)}
\left(\frac{u+\imunit v}{\sqrt{\al}}\right)
\\
&=
2^n
\enumber^{E_1(x,y,u,v)}
L_{m-1}^{(n)}(|u-x|^2+|v-y|^2),
\end{align*}
where $E_1(x,y,u,v)$ is the following expression:
\begin{align*}
E_1(x,y,u,v)
&=
-\frac{|u|^2+|v|^2}{2}
-\imunit\,\langle u,v\rangle
-\frac{|x|^2+|y|^2}{2}
+\imunit\,\langle x,y\rangle
+\langle u+\imunit v,x+\imunit y\rangle
\\
&=
-\frac{|u|^2-2\langle u,x\rangle+|x|^2}{2}
-\frac{|v|^2-2\langle
v,y\rangle+|y|^2}{2}
\\
&\qquad\qquad
-\imunit\,
\bigl(
\langle u,v\rangle
+\langle u,y\rangle
-\langle x,v\rangle
-\langle x,y\rangle
\bigr)
\\
&=
-\frac{|u-x|^2+|v-y|^2}{2}
-\imunit\,\langle u-x,v+y\rangle.
\end{align*}
Thereby we obtain~\eqref{eq:K_Hm}.
Finally, to compute the norm of the kernel, we evaluate it on the diagonal ($u=x$, $v=y$).
\end{proof}

\begin{cor}
\label{cor:UFH_K}
Let $x,y\in\bR^n$ and $z=x+\imunit y$.
Then
\begin{equation}
\label{eq:UFH_K}
U_{\cF_{\al,m}}^{\cH_m}
K_z^{\cF_{\al,m}}
=2^{-\frac{n}{2}}
\enumber^{\frac{\al}{2}|x|^2+\frac{\al}{2}|y|^2
-\imunit\al\langle x,y\rangle}
K_{\sqrt{\al}x,\sqrt{\al}y}^{\cH_m}.
\end{equation}
\end{cor}

\begin{proof}
It follows from~\eqref{def:U_Fm_Hm} and~\eqref{eq:K_Hm}.
Notice that~\eqref{eq:UFH_K},
jointly with~\eqref{eq:K_Fm_al_norm} and~\eqref{eq:K_Hm_norm},
can be used to test the isometric property of $U_{\cF_{\al,m}}^{\cH_m}$.
The right-hand side of
$U_{\cF_{\al,m}}^{\cH_m}$
has norm
$\sqrt{d_{n,m}}\enumber^{\frac{\al}{2}|z|^2}$,
which coincides with the norm of $K_z^{\cF_{\al,m}}$.
\end{proof}

\begin{cor}
For every $x,y,u,v$ in $\bR^n$,
\begin{equation}
\label{eq:K_Hm_via_sum_of_products_of_Laguerre_functions}
K_{x,y}^{\cH_m}(u,v)
= 2^n
\sum_{k\in J_{n,m}}\;
\prod_{r=1}^n 
\enumber^{-\imunit\,(u_r-x_r)(v_r+y_r)}\,
\ell_{k_r}((u_r-x_r)^2+(v_r-y_r)^2).
\end{equation}
\end{cor}

\begin{proof}
Follows from Propositions~\ref{prop:K_Hm} and~\ref{prop:Laguerre_decomposition}.
\end{proof}

\begin{rem}
\label{rem:Vasilevski_flatten_Fock}
Vasilevski~\cite[Section~2]{Vasilevski2000polyFock} used a multiplication by $\enumber^{-\frac{1}{2}|z|^2}$
to obtain a ``flattened'' version of $\cF_{1,m}$.
Unfortunately, this natural idea is not compatible with our scheme.
Indeed, consider the isometric isomorphism
$U_{\cF_{\al,m}}^{\cG_m}f\colon
\cF_{\al,m}\to\cG_m\le L^2(\bR^{2n},\tmu_{2n})$,
\begin{equation}
\label{def:U_Fm_Gm}
(U_{\cF_{\al,m}}^{\cG_m}f)(x,y)
\eqdef 
2^{\frac{n}{2}}
\enumber^{-\frac{1}{2}|x|^2
-\frac{1}{2}|y|^2}
f\left(\frac{x+\imunit y}{\sqrt{\al}}\right)
\qquad(x,y\in\bR^n).
\end{equation}
Define $\cG_m$ as the image of this operator.
Then $\cG_m$ is an RKHS, and its reproducing kernel is
\begin{equation}
\label{eq:K_Gm}
K^{\cG_m}_{x,y}(u,v)
=
2^n\,
\enumber^{-\frac{1}{2}(|u-x|^2+|v-y|^2)
-\imunit\,(\langle u-x,y+v\rangle+\langle x,y\rangle-\langle v,u\rangle)}
L_{m-1}^{(n)}(|u-x|^2+|v-y|^2).
\end{equation}
It is easy to see that 
$K_{x,y}^{\cG_m}(u,v)$, in general, does not coincide with
$K_{0,y}^{\cG_m}(u-x,v)$.
Therefore, by~\cite[Proposition~6.1]{HerreraMaximenkoRamos2022},
$\cG_m$ is not invariant with respect to horizontal translations.
So, the factor
$\enumber^{-\imunit \langle x,y\rangle}$
in the definition of $U_{\cF_{\al,m}}^{\cH_m}$ cannot be dropped.
\end{rem}

\begin{rem}[``Polyanalytic Steinwart--Hush--Scovel space'']
In~\cite{SteinwartHushScovel2006}, Steinwart, Hush, and Scovel gave an explicit description of the RKHS associated to the Gaussian kernel which is widely used in machine learning.
We introduce the following $m$-analytic analog of the space studied by Steinwart, Hush, and Scovel.
Let $\sigma>0$.
Define $\cS_{\si,m}$ as the space of all $m$-analytic functions $\bC^n\to\bC$ such that the following norm is finite:
\begin{equation}
\label{eq:def_Sm_norm}
\|f\|_{\cS_{\si,m}}
\eqdef 
\left(\frac{2^n \sigma^{2n}}{\pi^n}
\int_{\bC^n} |f(z)|^2
\exp\left(-4\sigma^2
\sum_{k=1}^n \Im(z_k)^2\right)
\dif{}\mu_{2n}(z) \right)^{1/2}.
\end{equation}
Furthermore, for $\al=2\sigma^2$ we define $U_{\cF_{\al,m}}^{\cS_{\si,m}}\colon\cF_{\al,m}\to\cS_m^{\sigma}$
and
$U_{\cS_{\si,m}}^{\cF_{\al,m}}\colon\cF_{\al,m}\to\cS_{\si,m}$
by
\begin{equation}
\label{eq:U_Fm_to_Sm}
U_{\cF_{\al,m}}^{\cS_{\si,m}}f(z)
\eqdef
\exp\left(-\sigma^2 \sum_{k=1}^n z_k^2\right)
f(z),
\qquad
(U_{\cS_m^\sigma}^{\cF_{\al,m}}
g)(z)
\eqdef\exp\left(\sigma^2 \sum_{k=1}^n z_k^2\right)
g(z).
\end{equation}
It is easy to see that
$U_{\cF_{\al,m}}^{\cS_m^\sigma}$
and $U_{\cS_m^\sigma}^{\cF_{\al,m}}$
are well-defined isometric isomorphisms of Hilbert spaces,
and they are mutually inverses.
By Proposition~\ref{prop:pushforward_RK},
$\cS_m^{\sigma}$ is a RKHS,
with the kernel:
\begin{equation}
\label{eq:Sm_kernel}
K_z^{\cS_m^{\sigma}}(w)
=
\exp\left(-\sigma^2
\sum_{k=1}^n (w_k-\overline{z_k})^2\right)
L_{m-1}^{(n)}\bigl(2\sigma^2|w-z|^2\bigr).
\end{equation}
Obviously,
$K_{z+a}^{\cS_m^{\sigma}}(w+a)
=K_z^{\cS_m^{\sigma}}(w)$
for every $a$ in $\bR^n$.
It follows by~\cite[Proposition~6.1]{HerreraMaximenkoRamos2022}
that $\cS_m^\sigma$ is invariant under horizontal translations.
We could use this space instead of $\cH_m$.
\end{rem}

\section{Horizontal Weyl operators and horizontal translations}
\label{sec:Weyl}

Given $a$ in $\bC^n$, we denote by $W_{a}$ the \emph{Weyl shift operator} acting in $L^2\left(\bC^n,\frac{\al^n}{\pi^n}\enumber^{-\al|z|^2}\mu_{2n}\right)$ by the rule
\[
(W_{a} f)(z)
\eqdef f(z-a)\,
\enumber^{\al\langle z,a\rangle-\frac{\al}{2}|a|^2}.
\]
These operators or their compressions to the analytic Fock space are considered in Folland~\cite[Chap. 2]{Folland1989} and Zhu~\cite[Section 2.6]{ZhuFock}.
It is known and easy to verify that $W_a$ is a unitary operator.
Given $f$ in $L^2\left(\bC^n,\frac{\al^n}{\pi^n}\enumber^{-\al|z|^2}\mu_{2n}\right)$,
the function $a\mapsto W_af$ is continuous (see a similar fact in~\cite[Proposition~2.41]{Folland1995}).

Moreover, for a fixed $a$ in $\bC^n$,
the function
$z\mapsto \enumber^{\al\langle z,a\rangle
-\frac{\al}{2}|a|^2}$
is analytic.
Since the elements of $\cF_{\al,m}$ have decomposition~\eqref{eq:polyanalytic_functions_decomposition}, we conclude that $\cF_{\al,m}$ is an invariant subspace with respect to $W_a$.
We denote by $W_{\al,m,a}$ the compression of $W_a$ to the subspace $\cF_{\al,m}$.
For each $a$ in $\bC^n$,
$W_{\al,m,a}$ is a unitary operator in $\cF_{\al,m}$.

In this paper,
we restrict ourselves to the ``horizontal'' Weyl operators in $\cF_{\al,m}$,
associated to the real translations,
associated to $a$ in $\bR^n$,
and denote them by $\rho_{\cF_{\al,m}}(a)$.
Additionally, we consider horizontal translations in $\cH_m$.
Formally, for every $a$ in $\bR^n$, we define the following operators.
\begin{itemize}
\item $\rho_{\cF_{\al,m}}(a)\colon\cF_{\al,m}\to\cF_{\al,m}$,
$\rho_{\cF_{\al,m}}(a)\eqdef W_{\al,m,a}$, i.e.,
\begin{equation}
\label{eq:horizontal_Weyl}
(\rho_{\cF_{\al,m}}(a) f)(z)
\eqdef f(z-a)\,\enumber^{\al\langle z,a\rangle-\frac{\al}{2}|a|^2}.
\end{equation}
\item $\rho_{\cH_m}(a)\colon\cH_m\to\cH_m$,
\begin{equation}
\label{eq:translation_Hm}
(\rho_{\cH_m}(a) f)(x,y)
\eqdef f(x-a,y).
\end{equation}
\end{itemize}

\begin{prop}
$(\rho_{\cF_{\al,m}},\cF_{\al,m})$
and $(\rho_{\cH_m},\cH_m)$
are unitary representation of $\bR^n$.
\end{prop}

\begin{proof}
We have already explained that for every $a$ in $\bR^n$,
$\rho_{\cF_{\al,m}}(a)$ is a unitary operator in $\cF_{\al,m}$.
The strong continuity of $\rho_{\cF_{\al,m}}$
and $\rho_{\cH_m}$
can be proved similarly to \cite[Proposition~2.41]{Folland1995}.

Directly from~\eqref{eq:K_Hm},
we observe that
$K^{\cH_m}_{x,y}(u,v)=K^{\cH_m}_{0,y}(u-x,v)$. By~\cite[Proposition~6.1]{HerreraMaximenkoRamos2022},
this implies that $\cH_m$
is invariant under the action of $\rho_{\cH_m}$.
The algebraic properties
\[
\rho_{\cF_{\al,m}}(a+b)=\rho_{\cF_{\al,m}}(a)\rho_{\cF_{\al,m}}(b),\qquad
\rho_{\cH_m}(a+b)=\rho_{\cH_m}(a)\rho_{\cH_m}(b)
\]
are easy to verify directly.
\end{proof}

\begin{prop}
\label{prop:U_rho_F_eq_rho_H_U}
For every $a$ in $\bR^n$,
\begin{equation}
\label{eq:U_Fm_Hm_intertwines}
U_{\cF_{\al,m}}^{\cH_m} \rho_{\cF_{\al,m}}(a)
=
\rho_{\cH_m}(\sqrt{\al}\,a) U_{\cF_{\al,m}}^{\cH_m}.
\end{equation}
\end{prop}
\begin{proof}
Let us prove~\eqref{eq:U_Fm_Hm_intertwines}.
Let $a\in\bR^n$ and $f\in\cF_{\al,m}$.
Following the rules from~\eqref{def:U_Fm_Hm} and~\eqref{eq:horizontal_Weyl}, we have
\begin{align*}
(U_{\cF_{\al,m}}^{\cH_m} \rho_{\cF_{\al,m}}(a)f)(x,y)
&=
2^{\frac{n}{2}}
\enumber^{-\frac{1}{2}|x|^2
-\frac{1}{2}|y|^2
-\imunit\,\langle x,y\rangle}
(\rho_{\cF_{\al,m}}(a) f)
\left(\frac{x+\imunit y}{\sqrt{\al}}\right)
\\
&=
2^{\frac{n}{2}}
\enumber^{-\frac{1}{2}|x|^2
-\frac{1}{2}|y|^2
-\imunit\,\langle x,y\rangle}
f\left(\frac{x+\imunit y}{\sqrt{\al}}-a\right)\enumber^{\al
\left\langle\frac{x+\imunit y}{\sqrt{\al}},a\right\rangle
-\frac{\al}{2}|a|^2}
\\
&=2^{\frac{n}{2}}
f\left(\frac{x+\imunit y}{\sqrt{\al}}-a\right)\enumber^{E_2(x,y)},
\end{align*}
where $E_2(x,y)$ is given by
\begin{align*}
E_2(x,y)
&=-\frac{1}{2}|x|^2
-\frac{1}{2}|y|^2
-\imunit\,\langle x,y\rangle
+\al
\left\langle\frac{x+\imunit y}{\sqrt{\al}},a\right\rangle
-\frac{\al}{2}|a|^2
\\
&=-\frac{1}{2}|x|^2
-\frac{1}{2}|y|^2
-\imunit\,\langle x-\sqrt{\al}a
+\sqrt{\al}a,y\rangle
+\langle x+\imunit y,\sqrt{\al}a\rangle
-\frac{\al}{2}|a|^2
\\
&=-\frac{1}{2}|x|^2
-\imunit \langle\sqrt{\al}a,y\rangle
+\sqrt{\al}\langle x+\imunit y,a\rangle
-\frac{\al}{2}|a|^2
-\frac{1}{2}|y|^2
-\imunit\,\langle x-\sqrt{\al}a,y\rangle
\\
&=-\frac{1}{2}|x|^2
-\imunit\langle\sqrt{\al}a,y\rangle
+\sqrt{\al}\langle x,a\rangle
+\sqrt{\al}\langle \imunit y,a\rangle
-\frac{\al}{2}|a|^2
-\frac{1}{2}|y|^2
-\imunit\,\langle x-\sqrt{\al}a,y\rangle
\\
&=
-\frac{1}{2}|x|^2
+\sqrt{\al}\langle x,a\rangle
-\frac{\al}{2}|a|^2
-\frac{1}{2}|y|^2
-\imunit\,\langle x
-\sqrt{\al}a,y\rangle
\\
&=
-\frac{1}{2}|x-\sqrt{\al}a|^2
-\frac{1}{2}|y|^2
-\imunit\,
\langle x-\sqrt{\al}a,y\rangle.
\end{align*}
So,
\begin{align*}
(U_{\cF_{\al,m}}^{\cH_m}
\rho_{\cF_{\al,m}}(a)f)(x,y)
&=2^{\frac{n}{2}}
f\left(\frac{x-\sqrt{\al}a+\imunit y}{\sqrt{\al}}\right)\enumber^{-\frac{1}{2}|x-\sqrt{\al}a|^2-\frac{1}{2}|y|^2
-\imunit\,\langle x-\sqrt{\al}a,y\rangle}
\\
&=
(U_{\cF_{\al,m}}^{\cH_m}f)(x-\sqrt{\al}a,y)
=
\bigl(\rho_{\cH_m}(\sqrt{\al}\,a) U_{\cF_{\al,m}}^{\cH_m}f\bigr)(x,y).
\qedhere
\end{align*}
\end{proof}

Proposition~\ref{prop:U_rho_F_eq_rho_H_U} means that for every $a$ in $\bR^n$,
the diagram on Figure~\ref{fig:U_rho_eq_rho_U} is commutative.
In other words,
$U_{\cF_{\al,m}}^{\cH_m}$ intertwines the representations $\rho_{\cF_{\al,m}}$
and $\rho_{\cH_m}$,
up to a dilation of the parameter.

\tikzset{spacenode/.style
={rounded corners, text centered,
text width=10ex, minimum size = 4ex,
draw, fill = lightgreen,
outer sep = 0.4ex}}

\tikzset{myedge/.style
={darkgreen, thick, -Stealth}}

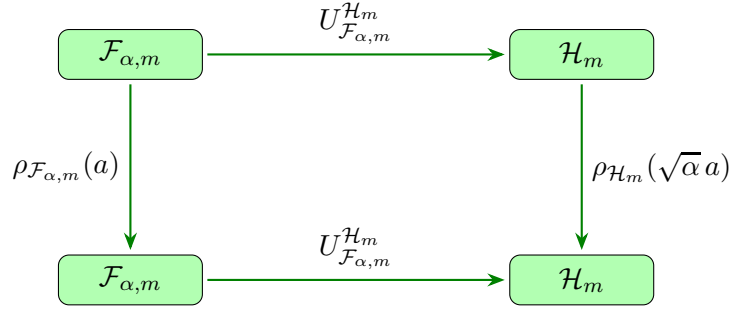
\begin{figure}[hbt]
\centering
\begin{tikzpicture}
\node[spacenode] (Fm1)
  at (-3, 0) {$\cF_{\al,m}$};
\node[spacenode] (Hm1)
  at (3, 0) {$\cH_m$};
\draw [myedge] (Fm1) edge(Hm1);
\node at (0, 0) [above]
  {$U_{\cF_{\al,m}}^{\cH_m}$};
\node[spacenode] (Fm2)
  at (-3, -3) {$\cF_{\al,m}$};
\node[spacenode] (Hm2)
  at (3, -3) {$\cH_m$};
\draw [myedge]
  (Fm2) edge (Hm2);
\node at (0, -3) [above]
  {$U_{\cF_{\al,m}}^{\cH_m}$};
\draw [myedge]
  (Fm1) edge (Fm2);
\node at (-3, -1.5) [left]
   {$\rho_{\cF_{\al,m}}(a)$};
\draw [myedge]
  (Hm1) edge (Hm2);
\node at (3, -1.5) [right]
  {$\rho_{\cH_m}(\sqrt{\al}\,a)$};
\end{tikzpicture}
\caption{
\label{fig:U_rho_eq_rho_U}
Isometric isomorphism
$U_{\cF_{\al,m}}^{\cH_m}$
and unitary operators
$\rho_{\cF_{\al,m}}(a)$,
$\rho_{\cH_m}(\sqrt{\al}\,a)$.
}
\end{figure}

\begin{example}
Let $z=x+\imunit y\in\bC^n$.
Then, by~\eqref{eq:horizontal_Weyl},
\eqref{eq:translation_Hm}, and
\eqref{eq:UFH_K},
\begin{align*}
(\rho_{\cF_{\al,m}}(a)K^{\cF_{\al,m}}_z)(w) 
&= \enumber^{\al\langle w,a \rangle - \frac{1}{2}|a|^2} K^{\cF_{\al,m}}_z(w-a)\\
&=\enumber^{\al\langle w,a \rangle - \frac{1}{2}|a|^2}
\enumber^{\al\langle w-a,z\rangle}
L_{m-1}^{(n)}(\al|(w-a)-z|^2)\\
&= \enumber^{-\al\langle a,z\rangle- \frac{1}{2}|a|^2} \enumber^{\al\langle w,z+a\rangle}
L_{m-1}^{(n)}(\al|w-(z+a)|^2)\\
&= \enumber^{-\al\langle a,z\rangle- \frac{1}{2}|a|^2} K_{z+a}^{\cF_{\al,m}}(w).
\end{align*}
Shortly, this means that
\[
\rho_{\cF_{\al,m}}(a) K_z^{\cF_{\al,m}}
=
\enumber^{-\al\langle a,z\rangle
-\frac{\al}{2}|a|^2}
K_{z+a}^{\cF_{\al,m}}.
\]
Combining with Corollary~\ref{cor:UFH_K}, we get
\[
U_{\cF_{\al,m}}^{\cH_m} \rho_{\cF_{\al,m}}(a) K_z^{\cF_{\al,m}}
= 2^{-\frac{n}{2}}
\enumber^{\frac{\al}{2}|x|^2+\frac{\al}{2}|y|^2
-\imunit\al
\langle x,y\rangle}
K_{\sqrt{\al}(x+a),\sqrt{\al}y}^{\cH_m}
=\rho_{\cH_m}(\sqrt{\al}\,a)
U_{\cF_{\al,m}}^{\cH_m}
K_z^{\cF_{\al,m}}.
\]
This calculation verifies Proposition~\ref{prop:U_rho_F_eq_rho_H_U}.
\end{example}

\section{Fourier connection between Laguerre and Hermite functions}
\label{sec:Fourier_connection_Laguerre_Hermite}

In this section, we recall a well-known connection between Laguerre and Hermite functions.
It is given in terms of the Fourier--Plancherel transform $F$ over the group $\bR$.

For $n$ in $\bNz$,
the $n$th Hermite function $\psi_n$ is defined by
\begin{equation}\label{def:Hermite_function}
\psi_n(t)
\eqdef(2^n\,n!\,\sqrt{\pi})^{-\frac{1}{2}}
\enumber^{-\frac{1}{2}t^2} H_n(t),
\end{equation}
where $H_n$ is the $n$th (physicist's) Hermite polynomial.
It is well known (see, e.g.,
\cite[Chapter~V]{Szego1975}) that $(\psi_n)_{n=0}^\infty$ is an orthonormal basis of $L^2(\bR)$, as well as
$(\ell_n)_{n=0}^\infty$ is an orthonormal basis of $L^2(\bR_+)$. These two bases are related in the following manner.

\begin{prop}\label{prop:main_formula}
Let $n\in\bNz$.
Then for every $a,u$ in $\bR$,
\begin{equation}\label{eq:Fourier_Laguerre_inverse}
\ell_n(u^2+a^2)
= \frac{1}{\sqrt{2}} \int_\bR \enumber^{\imunit u\xi} 
\psi_n\left(\frac{\xi+a}{\sqrt{2}}\right)
\psi_n\left(\frac{\xi-a}{\sqrt{2}}\right)
\dif{}\xi.
\end{equation}
Equivalently,
for every $a,\xi$ in $\bR$,
\begin{equation}
\label{eq:Fourier_Laguerre_function_nice}
\frac{1}{\sqrt{2\pi}}
\int_\bR \enumber^{-\imunit u\xi}
\ell_n(u^2+a^2)\,\dif{}u
=
\sqrt{\pi}\,
\psi_n\left(\frac{\xi+a}{\sqrt{2}}\right)
\psi_n\left(\frac{\xi-a}{\sqrt{2}}\right).
\end{equation}
\end{prop}

See different proofs of Proposition~\ref{eq:Fourier_Laguerre_inverse} in Folland~\cite[Theorem 1.104]{Folland1989} and Thangavelu~\cite[Theorem~1.3.4]{Thangavelu1993}.
In the remaining of this section,
we give another short proof of this proposition based on some tools from the theory of quantum harmonic oscillator~\cite[Section 1.1]{CombescureRobert2012book}.
Our proof is inspired by some ideas presented in~\cite[Section 1.1, Lemma 5]{CombescureRobert2012book}. 

To begin with, we recall the quantum operators momentum $P$ and position $Q$,
which are differential operators defined in some dense subspace of $L^2(\bR)$. For us, it will be enough to consider these operators acting as follows in $\cS(\bR)$, the space of
complex-valued Schwartz functions on $\bR$:
\[
(Pf)(t)
\eqdef
-\imunit\, \frac{\dif{}}{\dif{}t}f(t),
\qquad
(Qf)(t)
\eqdef
t\,f(t)
\qquad (f\in\cS(\bR)).
\]
The operators annihilation $\ba$ and creation $\ba^\dag$ are defined in terms of $P$ and $Q$ as
\[
\ba f
\eqdef
\frac{1}{\sqrt{2}} (Q + \imunit P)f,
\qquad
\ba^\dag f 
\eqdef
\frac{1}{\sqrt{2}} (Q - \imunit P)f
\qquad (f\in\cS(\bR)).
\]
These are the ladder operators for the quantum harmonic oscillator.
Moreover, they have the following properties.

\begin{itemize}
\item $P$ and $Q$ behave as symmetric operators on $\cS(\bR)$:
\[
\langle Pf,g \rangle = \langle f, Pg\rangle,
\qquad
\langle Qf,g \rangle = \langle f, Qg\rangle
\qquad (f,g\in\cS(\bR)).
\]
\item
$\ba$ and $\ba^\dag$
behave as mutually adjoint on $\cS(\bR)$:
\begin{equation}
\label{eq:ba_adjoint_badag}
\langle \ba f,g\rangle
= \langle f,\ba^\dag g\rangle
\qquad(f,g\in\cS(\bR)).
\end{equation}
\item These two pairs of operators satisfy the canonical commutation relations:
\begin{equation}\label{eq:canonical_commut_rel} 
[P,Q]f=-\imunit f,\qquad
[\ba,\ba^\dag]f = f \qquad (f\in\cS(\bR)).
\end{equation}
\item Annihilation and creation operators applied to the $n$th Hermite function give the following recurrence formulas called
Segal--Bargmann representation of the canonical commutation relation:
\begin{equation}\label{eq:Hermite_creation_anihilation_rec}
\ba \psi_n 
= \sqrt{n}\,\psi_{n-1},
\qquad
\ba^\dag \psi_n
= \sqrt{n+1}\,\psi_{n+1}.
\end{equation}
\item Given $x$ in $\bR$ and $f$ in $\cS(\bR)$, we have the following representation of the translation and modulation operators in terms of $P$ and $Q$:
\begin{equation}\label{eq:translation_modulation_P_Q}
f(t+x) = \bigl( \enumber^{\imunit x P}f\bigr)(t),
\qquad
\enumber^{\imunit xt}f(t) = \big(\enumber^{\imunit xQ}f\big)(t).
\end{equation}
\item Let $z$ be any complex number. The operators $\enumber^{z\ba^\dag}$ and $\enumber^{\conj{z}\ba}$ are densely defined in $\cS(\bR)$, where they interact in the following manner:
    \[
    \langle f, \enumber^{z\ba^\dag}g\rangle
    = \langle \enumber^{\conj{z}\ba}f, g\rangle 
    \qquad (f,g\in\cS(\bR)).
    \]
\end{itemize}
\begin{lem}
For every $z=u+\imunit v$ in $\bC$, define the following operator in $\cS(\bR)$:
\[
\rho(z)\eqdef
\enumber^{-\frac{1}{2}|z|^2}\enumber^{z\ba^\dag}\enumber^{-\conj{z}\ba}.
\]
Then
\begin{equation}\label{eq:rep_rho_eval}
(\rho(u+\imunit v)f)(t) 
= \enumber^{-\imunit (u-\sqrt{2}t)v} f(t-\sqrt{2}u).
\end{equation}
\end{lem}    
\begin{proof}
Let $z=u+\imunit v$ in $\bC^n$ and $f$ be a function in $\cS(\bR)$. Then
\[
(\rho(u+\imunit v)f)(t)
= \left(\enumber^{-\frac{1}{2}|z|^2}\enumber^{z\ba^\dag}\enumber^{-\conj{z}\ba}f\right)(t)
= \left(\enumber^{z\ba^\dag - \conj{z}\ba}f\right)(t).
\]
The second equality holds by using the Baker--Campbell--Hausdorff formula (see, for example~\cite{BonfiglioliFulci2011,CombescureRobert2012book}) 
and the canonical commutation relations~\eqref{eq:canonical_commut_rel}.
It is easy to see that $z\ba^\dag - \conj{z}\ba$ is equivalent to $-\sqrt{2}\imunit uP+ \sqrt{2}\imunit vQ$. 
Substituting, applying~\eqref{eq:translation_modulation_P_Q} and the Baker--Cambell--Hausdorff formula again, we conclude~\eqref{eq:rep_rho_eval}.
\end{proof}

Let us now turn to the proof of the main result of this section.
\begin{proof}[Proof of Proposition~\ref{prop:main_formula}]
We compute the inner product $\langle \psi_n, \rho(z)\psi_n \rangle$ in two different ways and then we compare one to the other. First,
expanding the exponential operators and applying repeatedly~\eqref{eq:Hermite_creation_anihilation_rec} we get
\begin{align*}
    \langle \psi_n, \rho(z)\psi_n \rangle
    &= \enumber^{-\frac{1}{2}|z|^2} \langle \enumber^{\conj{z} \ba} \psi_n,
    \enumber^{-\conj{z}\ba} \psi_n \rangle\\
    &= \enumber^{-\frac{1}{2}|z|^2} \sum_{k=0}^n \sum_{j=0}^n
    \sqrt{\binom{n}{k}\binom{n}{j}k!j!}\, 
    \frac{(\conj{z})^k(-z)^j}{(k!)(j!)}
    \langle \psi_{n-k}, \psi_{n-j} \rangle.
\end{align*}
Since Hermite functions are orthonormal, the double sum collapses to the sum over the diagonal,
which can be expressed in terms of the Laguerre function~\eqref{eq:Laguerre_function_def}:
\[
\langle \psi_n, \rho(z)\psi_n \rangle
=
\enumber^{-\frac{1}{2}|z|^2} \sum_{k=0}^n
\binom{n}{k}\, 
\frac{(-1)^k|z|^{2k}}{k!}
= \ell_n(|z|^2).
\]
So, on the one hand,
for $z = a-\imunit u$, we have 
\begin{equation}\label{eq:proof_main_formula_LHS}
\langle \psi_n, \rho(a-\imunit u)\psi_n \rangle 
= \ell_n(u^2+a^2).
\end{equation}
On the other hand, by~\eqref{eq:rep_rho_eval},
\[
\langle \psi_n, \rho(a-\imunit u)\psi_n \rangle
= \int_\bR \psi_n(\eta)\,\enumber^{\imunit ua}
\enumber^{\sqrt{2}\imunit u(\eta-\sqrt{2}a)}\psi_n(\eta-\sqrt{2}a)\dif{}\eta.
\]
Setting $\xi+a = \sqrt{2}\eta$, we get 
\begin{equation}\label{eq:proof_main_formula_RHS}
    \langle \psi_n, \rho(a-\imunit u)\psi_n \rangle
    = \frac{1}{\sqrt{2}}\enumber^{\imunit ua} 
    \int_\bR \enumber^{\imunit  u(\xi-a)}
    \psi_n\left(\frac{\xi+a}{\sqrt{2}}\right)\,
    \psi_n\left(\frac{\xi-a}{\sqrt{2}}\right)\dif{}\xi.
\end{equation}
Finally, simplifying the exponential part and comparing~\eqref{eq:proof_main_formula_LHS} and~\eqref{eq:proof_main_formula_RHS}, we obtain~\eqref{eq:Fourier_Laguerre_inverse}.
\end{proof}

\section{Horizontal Fourier transform of the reproducing kernel}
\label{sec:Fourier_transform_of_reproducing_kernel}

Following the scheme from~\cite{HerreraMaximenkoRamos2022}, now we treat the domain $\bR^{2n}$ as the product $G\times Y$, where
\begin{itemize}
\item $G$ is the group $\bR^n$ provided with the measure
$\tmu_n \eqdef (2\pi)^{-n/2}\mu_n$,
\item $Y$ is the measure space $(\bR^n,\tmu_n)$.
\end{itemize}
The dual group $\widehat{\bR^n}$
is identified with $\bR^n$ using the pairing
$(x,\xi)\mapsto\enumber^{\imunit \langle x,\xi\rangle}$.
The dual Haar measure is also $\tmu_n$.
Let $F$ denote the Fourier--Plancherel transform
$L^2(\bR^n,\tmu_n)\to L^2(\bR^n,\tmu_n)$.
We denote by $L_{\cdot,y}(v)$
the Fourier transform of $K^{\cH_m}_{0,y}(\cdot,v)$:
\[
L_{\xi,y}(v)
\eqdef (F (K^{\cH_m}_{0,y}(\cdot,v)))(\xi),
\]
i.e.,
\begin{equation}
\label{eq:def_L}
L_{\xi,y}(v)
\eqdef
\frac{1}{(2\pi)^{n/2}}
\int_{\bR^n} K^{\cH_m}_{0,y}(u,v)
\,\enumber^{-\imunit \langle u,\xi\rangle}\,\dif\mu_n(u).
\end{equation}

\begin{thm}
\label{thm:L_formula}
For every $\xi,y,v$ in $\bR^n$,
\begin{equation}
\label{eq:L_formula}
L_{\xi,y}(v)
=
\sum_{k\in J_{n,m}}\;
\overline{q_{k,\xi}(y)}
q_{k,\xi}(v)
=
\sum_{k\in J_{n,m}}\;
q_{k,\xi}(y)
q_{k,\xi}(v),
\end{equation}
where
\begin{equation}
\label{eq:q_formula}
q_{k,\xi}(v)
\eqdef 
2^{n/2}\pi^{n/4}
\prod_{r=1}^n
\psi_{k_r}\left(\frac{\xi_r+2v_r}{\sqrt{2}}\right).
\end{equation}
\end{thm}

\begin{proof}
Notice that
\[
K_{0,y}^{\cH_m}(u,v)
= 2^n
\sum_{k\in J_{n,m}}\;
\prod_{r=1}^n 
\enumber^{-\imunit u_r(v_r+y_r)}\,
\ell_{k_r}(u_r^2+(v_r-y_r)^2).
\]
We apply the Fourier transform to this function and denote by $I_{\xi,y,v,k}$ the $k$th summand:
\[
L_{\xi,y}(v)
=
\sum_{k\in J_{n,m}}\;
I_{\xi,y,v,k},
\]
where
\[
I_{\xi,y,v,k}
\eqdef
\frac{2^n}{(2\pi)^{n/2}}
\int_{\bR^n}
\enumber^{-\imunit\, \langle u,\xi\rangle}
\left(
\prod_{r=1}^n
\enumber^{-\imunit u_r(v_r+y_r)}\,
\ell_{k_r}(u_r^2+(v_r-y_r)^2)
\right)\,\dif\mu_n(u).
\]
By Fubini's theorem, the integral over $\bR^n$ decomposes into a product of $n$ integrals:
\[
I_{\xi,y,v,k}
=
\prod_{r=1}^n
\left(
\frac{2}{\sqrt{2\pi}}
\int_{\bR}
\enumber^{-\imunit u_r (\xi_r+v_r+y_r)}
\ell_{k_r}(u_r^2+(v_r-y_r)^2)
\,\dif{}u_r\right).
\]
Finally, we apply~\eqref{eq:Fourier_Laguerre_function_nice} with $v_r-y_r$ instead of $a$ and $\xi_r+y_r+v_r$ instead of $\xi$:
\begin{align*}
I_{\xi,y,v,k}
&=
\prod_{r=1}^n
\left(
2\sqrt{\pi}\,
\psi_{k_r}\left(\frac{\xi_r+y_r+v_r+v_r-y_r}{\sqrt{2}}\right)
\psi_{k_r}\left(\frac{\xi_r+y_r+v_r-v_r+y_r}{\sqrt{2}}\right)
\right)
\\
&=
\prod_{r=1}^n
\left(
2\sqrt{\pi}\,
\psi_{k_r}\left(\frac{\xi_r+2v_r}{\sqrt{2}}\right)
\psi_r\left(\frac{\xi_r+2y_r}{\sqrt{2}}\right)
\right)
=
q_{k,\xi}(y) q_{k,\xi}(v).
\qedhere
\end{align*}
\end{proof}

\begin{prop}\label{prop:q_orthonormal_in_L2}
For every $\xi$ in $\bR^n$,
$(q_{k,\xi})_{|k|\le m-1}$
is an orthonormal list of functions in $L^2(\bR^n,\tmu_n)$.
\end{prop}

\begin{proof}
Let $k,p\in\bNz^n$.
By the orthonormality of the Hermite functions,
\begin{align*}
\langle q_{k,\xi},q_{p,\xi}\rangle_{L^2(\bR^n,\tmu_n)}
&=
\frac{2^n\pi^{n/2}}{(2\pi)^{n/2}}
\prod_{r=1}^n
\left(
\int_{\bR}
\psi_{k_r}\left(\frac{\xi_r+2v_r}{\sqrt{2}}\right)
\psi_{p_r}\left(\frac{\xi_r+2v_r}{\sqrt{2}}\right)
\,\dif{}v_r
\right)
\\
&=
\prod_{r=1}^n
\int_\bR
\psi_{k_r}(t_r)
\psi_{p_r}(t_r)\,
\dif{}t_r
=
\prod_{r=1}^n
\de_{k_r,p_r}
=\de_{k,p}.
\qedhere
\end{align*}
\end{proof}

\section{Horizontal Fourier decomposition of the spaces}
\label{sec:decomposition_of_spaces}

In this section we decompose the spaces $\cH_m$ and $\cF_{\al,m}$ into finitely many copies of $L^2(\bR^n,\tmu_n)$, following the scheme from~\cite{HerreraMaximenkoRamos2022}.

Recall $J_{n,m}$ is defined in~\eqref{eq:J_def}.
From now on, we briefly write $d$ instead of $d_{n,m}$.
We enumerate the elements of $J_{n,m}$ in the lexicographic order. Let $\phi\colon\{1,\ldots,d\}\to J_{n,m}$ be the corresponding bijection.
For example, for $n=2$ and $m=3$,
\begin{align*}
&\phi(1)=(0,0),\quad
\phi(2)=(0,1),\quad
\phi(3)=(0,2),\\
&\phi(4)=(1,0),\quad
\phi(5)=(1,1),\quad
\phi(6)=(2,0).
\end{align*}
Hence we can write
\begin{equation}\label{eq:L_new_form}
L_{\xi,y}(v) 
= \sum_{k\in J_{n,m}} q_{k,\xi}(y)q_{k,\xi}(v)
=\sum_{j=1}^d q_{\phi(j),\xi}(y)q_{\phi(j),\xi}(v).
\end{equation}

Let $I$ be the identity operator in $L^2(\bR^n,\tmu_n)$.
We denote by $\widehat{\cH}_m$ the image of $\cH_m$ under the operator $F\otimes I$.
This is a closed subspace of $L^2(\bR^{2n},\tmu_{2n})$.

For every $\xi$ in $\bR^n$, we define
\[
\widehat{\cH}_{m,\xi} 
\eqdef \operatorname{span}\big(\{q_{\phi(j),\xi}\::\:j=1,\ldots,d\}\big).
\]
By Proposition~\ref{prop:q_orthonormal_in_L2}, the list $(q_{\phi(j),\xi})_{j=1}^d$ is an orthonormal basis of $\widehat{\cH}_{m,\xi}$.
This space is an RKHS 
whose reproducing kernel is $(L_{\xi,y})_{y\in\bR^n}$ given by~\eqref{eq:L_new_form}. 

According to~\cite{HerreraMaximenkoRamos2022}, 
$\widehat{\cH}_m$ decomposes into the direct integral
\[
\widehat{\cH}_m 
=
\int_{\bR^n}^\oplus \widehat{\cH}_{m,\xi}\, \dif\tmu_n(\xi).
\]
Every ``fiber''
$\widehat{\cH}_{m,\xi}$
has dimension $d$.
Using the orthonormal basis $(q_{\phi(j)},\xi)_{j=1}^d$,
we can construct an isometric isomorphism
$\widehat{\cH}_{m,\xi}\to\bC^d$
in a standard way.
Then, 
\[
\int_{\bR^n}^\oplus \bC^d \,\dif\tmu_n(\xi)
= L^2(\bR^n,\,\tmu_n)\otimes \bC^d
= L^2(\bR^n, \tmu_n, \bC^d)
= L^2(\bR^n,\,\tmu_n)^d.
\]
We identify $L^2(\bR^n,\tmu_n,\bC^d)$ with $L^2(\bR^n,\tmu_n)^d$ in the obvious way.
Define $N\colon\widehat{\cH}_m\to L^2(\bR^n,\,\widetilde{\mu}_n)^d$ 
as follows: for every $j$ in $\{1,\ldots,d\}$,
\[
(N g)(\xi)_j
= \langle g(\xi, \cdot), q_{\phi(j),\xi}\rangle_{L^2(\bR^n,\tmu_n)}
= \frac{1}{(2\pi)^{n/2}}\int_{\bR^n} g(\xi,v) q_{\phi(j),\xi}(v)\, \dif{}\mu_n(v)
\quad
(\xi\in\bR^n).
\]
Then $N$ is an isometric isomorphism.
Its adjoint (inverse) operator is
\[
(N^\ast h)(\xi, y) = \sum_{j=1}^{d} q_{\phi(j),\xi}(y)h_j(\xi)
\qquad
(h\in L^2(\bR^n,\tmu_n)^d). 
\]
In what follows, we construct isometric isomorphisms from $\cH_m$ and $\cF_{\al,m}$ onto
$L^2(\bR^n,\,\tmu_n)^d$ using~\eqref{eq:q_formula}.
First, let $R_{\cH_m}\colon\cH_m\to L^2(\bR^n,\,\tmu_n)^d$ be defined as $R_{\cH_m} \eqdef N(F\otimes I)$. 

\begin{prop}
$R_{\cH_m}$ is an isometric isomorphism.
If $g\in\cH_m$ and $\|g\|_{L^1(\bR^{2n})}<+\infty$, then
\begin{equation}\label{eq:R_Hm_explicit}
(R_{\cH_m}g)(\xi)_j
=
\frac{1}{2^{n/2}\pi^{3n/4}}\,\int_{\bR^{2n}} g(u,v) 
\enumber^{-\imunit \langle u, \xi\rangle}\,
\prod_{r=1}^n
\psi_{\phi(j)_r}\left(\frac{\xi_r+2v_r}{\sqrt{2}}\right)
\dif{}\mu_{2n}(u,v).
\end{equation}
\end{prop}
\begin{proof}
    In our scheme, for every $j$ in $\{1,\ldots,d\}$,
    \[
    (R_{\cH_m}g)(\xi)_j
    = \frac{1}{(2\pi)^{n/2}}\int_{\bR^n}\left( 
    \frac{1}{(2\pi)^{n/2}}\int_{\bR^n} g(u,v)\enumber^{-\imunit \langle u, \xi\rangle}\dif\mu_n(u) \right) \conj{q_{\phi(j),\xi}(v)}\,\dif\mu_n(v).
    \qedhere
    \]
\end{proof}

Additionally, we define $R_{\cF_{\al,m}}\colon\cF_{\al,m}\to L^2(\bR^n,\,\tmu_n)^d$ as $R_{\cF_{\al,m}}\eqdef R_{\cH_m}U_{\cF_{\al,m}}^{\cH_m}$,
see~\eqref{def:U_Fm_Hm}.

\begin{prop}
$R_{\cF_{\al,m}}$ is an isometric isomorphism.
If $f\in\cF_{\al,m}$ and
\[
\int_{\bC^n}|f(z)| \enumber^{-\frac{1}{2}|z|^2} \dif\mu_{2n}(z) < +\infty,
\]
then, for every $j$ in $\{1,\ldots,d\}$,
\begin{equation}\label{eq:R_Fm_explicit}
(R_{\cF_{\al,m}}f)(\xi)_j 
= \frac{1}{\pi^{3n/4}}\,\int_{\bR^{2n}} 
    f\left(\frac{u+\imunit v}{\al}\right) 
    \enumber^{-\frac{1}{2}|u|^2-\frac{1}{2}|v|^2+\imunit\langle u,v-\xi \rangle}\, 
    \prod_{r=1}^n \psi_{\phi(j)_r}\left(\frac{\xi_r+2v_r}{\sqrt{2}}\right)
    \dif{}\mu_{2n}(u,v).
\end{equation}
\end{prop}
\begin{proof}
Substituting $U_{\cF_{\al,m}}^{\cH_m}f$ instead of $g$ in~\eqref{eq:R_Hm_explicit}, we get
\[
(R_{\cH_m}U_{\cF_{\al,m}}^{\cH_m}f)(\xi)_j
= \frac{2^{n/2}\pi^{n/4}}{(2\pi)^n}\,
    \int_{\bR^{2n}} (U_{\cF_{\al,m}}^{\cH_m}f)(u,v) 
    e^{-\imunit \langle u, \xi\rangle}\,\prod_{r=1}^n
    \psi_{\phi(j)_r}\left(\frac{\xi_r+2v_r}{\sqrt{2}}\right)
    \dif{}\mu_{2n}(u,v). \qedhere
\]
\end{proof}
For $n=1$ and $\al=1$, the dimension $d$ reduces to $m$ and formula~\eqref{eq:R_Fm_explicit} reads as
\[
(R_{\cF_{1,m}}f)(\xi)_j 
= \frac{1}{\pi^{3/4}}\,\int_{\bR^2} 
    f(u+\imunit v) 
    \enumber^{-\frac{1}{2}|u|^2-\frac{1}{2}|v|^2+\imunit u(v-\xi)}\, 
    \psi_j\left(\frac{\xi+2v}{\sqrt{2}}\right)
    \dif{}\mu_2(u,v).
\]
This formula is similar to the one constructed in~\cite[Section 2]{SanchezGonzalezLopezArroyo2018}.

\begin{example}
\label{example:R_K}
Let $y\in\bR^n$.
We consider
$K_{\imunit y}^{\cF_{\al,m}}\in\cF_{\al,m}$.
Let us compute the corresponding functions in $\cH_m$,
$\widehat{\cH}_m$, and $L^2(\bR^n,\tmu_n)^d$.
By~\eqref{eq:UFH_K},
\[
U_{\cF_{\al,m}}^{\cH_m}
K_{\imunit y}^{\cF_{\al,m}}
= 2^{-\frac{n}{2}}
\enumber^{\frac{\al}{2}|y|^2}K_{0,\sqrt{\al} y}^{\cH_m}.
\]
Next, by Theorem~\ref{thm:L_formula},
\begin{align*}
((F\otimes I) U_{\cF_{\al,m}}^{\cH_m}
K_{\imunit y}^{\cF_{\al,m}})(\xi,v)
&=
2^{-\frac{n}{2}}
\enumber^{\frac{\al}{2}|y|^2}
(F K_{0,\sqrt{\al}y}^{\cH_m}(\cdot,v))(\xi)
\\[1ex]
&=
2^{-\frac{n}{2}}
\enumber^{\frac{\al}{2}|y|^2}
L_{\xi,\sqrt{\al}y}
=
2^{-\frac{n}{2}}
\enumber^{\frac{\al}{2}|y|^2}
\sum_{k\in J_{n,m}}
q_{k,\xi}(\sqrt{\al}y)
q_{k,\xi}(v).
\end{align*}
Finally, applying $N$ we get
\[
(R_{\cF_{\al,m}}
K_{\imunit y}^{\cF_{\al,m}})(\xi)
=
2^{-\frac{n}{2}}
\enumber^{\frac{\al}{2}|y|^2}
\Bigl[
q_{\phi(j),\xi}(\sqrt{\al}\,y)
\Bigr]_{j=1}^d.
\]
The squared norm of the obtained vector-function is
\begin{align*}
\left\|
R_{\cF_{\al,m}} K_{\imunit y}^{\cF_{\al,m}}\right\|_{L^2(\bR^n,\tmu_n)^d}^2
&=
2^{-n}
\enumber^{\al|y|^2}
\sum_{j=1}^d
\int_{\bR^n} |q_{\phi(j),\xi}(\sqrt{\al}y)|^2\,\dif\tmu_n(\xi)
\\
&=
\enumber^{\al|y|^2}
\sum_{j=1}^d
\prod_{r=1}^n
\int_\bR
\left|\psi_{\phi(j)_r}\left(\frac{\xi_r
+2\sqrt{\al}y_r}{\sqrt{2}}\right)\right|^2\,
\frac{\dif{}\xi_r}{\sqrt{2}}
=d\,\enumber^{\al|y|^2},
\end{align*}
which coincides with
$\|K_{\imunit y}^{\cF_{\al,m}}\|_{\cF_{\al,m}}^2$,
according to~\eqref{eq:K_Fm_al_norm}.
\end{example}

\begin{example}
Let $y\in\bR^n$
and $j_0\in\{1,\ldots,d\}$.
Consider
$\be\eqdef \phi(j_0)+\bone_n$,
and
\[
f(w)
\eqdef K^{\cF^\al_{(\be)}}_{\imunit y}(w)
= \prod_{r=1}^n \enumber^{\al
\langle w,\imunit y\rangle}
L_{\phi(j_0)_r}(\al|w_r-\imunit y_r|^2),
\]
see~\eqref{eq:K_true_polyFock_multiindex}.
Similarly to Example~\ref{example:R_K},
\[
((F\otimes I) U_{\cF_{\al,m}}^{\cH_m}
K_{\imunit y}^{\cF^\al_{(\be)}})(\xi,v)
=
2^{-\frac{n}{2}}
\enumber^{\frac{\al}{2}|y|^2}
q_{\phi(j_0),\xi}(\sqrt{\al}y)
q_{\phi(j_0),\xi}(v).
\]
Therefore,
\[
(R_{\cF_{\al,m}}
K_{\imunit y}^{\cF^\al_{(\be)}})(\xi)
=
2^{-\frac{n}{2}}
\enumber^{\frac{1}{2}|y|^2}
q_{\phi(j_0),\xi}(\sqrt{\al}\,y)\,
\left[
\de_{j,j_0}\,
\right]_{j=1}^d.
\]
In other words, only one of the components of the vector-function takes non-trivial values, and the others are zero.
\end{example}

\begin{example}
Let $h\in L^1(\bR^n,\tmu_n)$ and fix $y$ in $\bR^n$.
Consider $f\colon\bC^n\to \bC$ defined by
\[
f(w) 
= \int_{\bR^n} h(x)\enumber^{-\frac{1}{2}|x|^2+\imunit\langle x,y\rangle} K^{\cF_{\al,m}}_{\frac{1}{\sqrt{\al}}(x+\imunit y)}(w)\,\dif\tmu_n(x).
\]
This integral can be understood as a Bochner integral of an integrable function with values in $\cF_{\al,m}$.
Therefore, $f$ belongs $\cF_{\al,m}$. Furthermore, we apply $U_{\cF_{\al,m}}^{\cH_m}$ and we pass it under the integral sign. Using~\eqref{eq:UFH_K}, we get 
\begin{align*}
(U_{\cF_{\al,m}}^{\cH_m} f)(u,v)
&= \int_{\bR^n} h(x)\enumber^{-\frac{1}{2}|x|^2+\imunit\langle x,y\rangle} \left(U_{\cF_{\al,m}}^{\cH_m}K^{\cF_{\al,m}}_{\frac{1}{\sqrt{\al}}(x+\imunit y)}\right)(w)\,\dif\tmu_n(x)\\
&= 2^{-\frac{n}{2}}\enumber^{\frac{1}{2}|y|^2}\int_{\bR^n} 
h(x)
K^{\cH_m}_{0,y}(u-x,v)\,\dif \tmu_n(x).
\end{align*}
Hence, the image of $f$ in $\cH_m$ is the following convolution:
\[
(U_{\cF_{\al,m}}^{\cH_m} f)(u,v)
= 2^{-\frac{n}{2}}\enumber^{\frac{1}{2}|y|^2}
\left(h(\,\cdot\,)
\ast\, K_{0,y}^{\cH_m}(\,\cdot\,,v)\right)(u).
\]
Applying the convolution theorem, we get the image of $f$ in $\widehat{\cH}_m$:
\[
((F\otimes I) U_{\cF_{\al,m}}^{\cH_m}  f)(\xi,v)
= 2^{-\frac{n}{2}} \enumber^{\frac{1}{2}|y|^2} 
\widehat{h}(\xi) L_{\xi,y}(v)
= 2^{-\frac{n}{2}}\enumber^{\frac{1}{2}|y|^2} 
\widehat{h}(\xi) \sum_{j=1}^d q_{\phi(j),\xi}(y)q_{\phi(j),\xi}(v).
\]
Finally, $f$ is converted into the following vector-function in $L^2(\bR^n,\tmu_n)^d$ by applying the operator $N$:
\[
(R_{\cF_{\al,m}}f)(\xi)
= 2^{-\frac{n}{2}}\enumber^{\frac{1}{2}|y|^2} 
\widehat{h}(\xi)\left[q_{\phi(j),\xi}(y)\right]_{j=1}^d.
\]
\end{example}

\section{Von Neumann algebra of translation-invariant operators}
\label{sec:algebra_of_translation_invariant_operators}

In this section, we obtain a description
of the von Neumann algebras $\cC(\rho_{\cH_m})$ 
and $\cC(\rho_{\cF_{\al,m}})$.

\begin{prop}
\label{prop:U_and_centralizers}
$\cC(\rho_{\cF_{\al,m}})$ and
$\cC(\rho_{\cH_m})$ are spatially isomorphic.
More precisely,
\[
\cC(\rho_{\cH_m})
=\bigl\{
S\in\cB(\cH_m)\colon\quad
U_{\cH_m}^{\cF_{\al,m}}
S
U_{\cF_{\al,m}}^{\cH_m}
\in\cC(\rho_{\cF_{\al,m}})
\bigr\}.
\]
\end{prop}

\begin{proof}
Follows from Proposition~\ref{prop:U_rho_F_eq_rho_H_U}.
\end{proof}

We identify $L^\infty(\bR^n)^{d\times d}$ with the W*-algebra $L^\infty(\bR^n,\cM_d)$ of $d\times d$ matrix-functions defined on $\bR^n$ and measurable with respect to $\mu_n$ (or $\tmu_n$).
Given $\si$ in $L^\infty(\bR^n)^{d\times d}$,
we denote by $M_\si$ the multiplication operator acting in $L^2(\bR^n,\tmu_n)^d$ by
\begin{equation}
\label{eq:mul_operator_def}
(M_\si f)(\xi) \eqdef \si(\xi) f(\xi)
\qquad(\xi\in\bR^n).
\end{equation}
In the right-hand side of~\eqref{eq:mul_operator_def},
$\si(\xi)\in\cM_d=\bC^{d\times d}$ and $f(\xi)\in\bC^d$
for every $\xi$ in $\bR^n$.

The W*-algebra 
$L^\infty(\bR^n)^{d\times d}$ 
is isometrically isomorphic to the
von Neumann algebra
$\bigl\{M_\si\colon\ \si\in L^\infty(\bR^n)^{d\times d}\bigr\}$.

We denote by $I_d$ the identity matrix of order $d$.
The set $L^\infty(\bR^n)\otimes I_d$ is a commutative subalgebra of 
$L^\infty(\bR^n)^{d\times d}$.
Given $h$ in $L^\infty(\bR^n)$,
we denote by $h I_d$ the scalar matrix-function
\[
(h I_d)(\xi) \eqdef h(\xi) I_d\qquad(\xi\in\bR^n),
\]
which corresponds to the scalar multiplication operator:
\[
(M_{h I_d}f)(\xi)
= h(\xi) f(\xi)
\qquad(\xi\in\bR^n).
\]
We define
$\La_{\cH_m}\colon
L^\infty(\bR^n)^{d\times d}
\to
\cC(\rho_{\cH_m})$
and
$\La_{\cF_{\al,m}}\colon
L^\infty(\bR^n)^{d\times d}
\to
\cC(\rho_{\cF_{\al,m}})$ by
\[
\La_{\cH_m}(\si)
=
R_{\cH_m}^\ast M_\si R_{\cH_m},
\qquad
\La_{\cF_{\al,m}}(\si)
=
R_{\cF_{\al,m}}^\ast M_\si R_{\cF_{\al,m}}.
\]

\begin{thm}
\label{thm:Neumann_algebras}
$\La_{\cH_m}$
and $\La_{\cF_{\al,m}}$
are well-defined isometric isomorphisms.
Hence,
$\cC(\rho_{\cF_{\al,m}})$
and $\cC(\rho_{\cH_m})$
are isometrically isomorphic
to $L^\infty(\bR^n)^{d\times d}$.
\end{thm}

\begin{proof}
The part about $\La_{\cH_m}$
follows from~\cite{HerreraMaximenkoRamos2022}
and Theorem~\ref{thm:L_formula}.
The part about $\La_{\cF_{\al,m}}$ is a corollary which takes into account
Proposition~\ref{prop:U_and_centralizers}.
\end{proof}

The isometric isomorphisms $\La_{\cH_m}^{-1}$ and $\La_{\cF_{\al,m}}^{-1}$ can be called ``spectral symbols'' of the corresponding von Neumann algebras.

In particular, Theorem~\ref{thm:Neumann_algebras} implies that for $m\ge 2$, algebras $\cC(\rho_{\cF_{\al,m}})$ and $\cC(\rho_{\cH_m})$ are not commutative.

In the rest of this section,
we compute $\La_{\cF_{\al,m}}^{-1}(S)$
for some examples:
horizontal Weyl translations,
vertical Toeplitz operators,
and horizontal convolutions.

Given $b$ in $\bR^n$, we define
$E_b\colon\bR^n\to\bC$ by
\[
E_b(\xi)
\eqdef
\enumber^{\imunit \langle \xi,b\rangle}\qquad(\xi\in\bR^n).
\]
In other words, $E_b$ is the character of the group $\bR^n$ naturally associated to $b$,
and $b\mapsto E_b$ is an isomorphism between $\bR^n$ and its dual group.

\begin{example}
\label{example:diagonalization_of_Weyl_translations}
Let $a$ in $\bR^n$. By Proposition~\ref{prop:U_rho_F_eq_rho_H_U}, the operator $U_{\cF_{\al,m}}^{\cH_m}$
intertwines $\rho_{\cF_{\al,m}}(a/\sqrt{\al})$ and $\rho_{\cH_m}(a)$.
Moreover, it is easy to see that $F\otimes I$ intertwines the translation operator $\rho_{\cH_m}(a)$ in $\cH_m$ and the modulation operator in $\widehat{\cH}_m$:
\[
(F\otimes I)\rho_{\cH_m}(a)
=\enumber^{-\imunit\langle \xi,a\rangle}(F\otimes I).
\]
Then, every translation operator $\rho_{\cH_m}(a)$ is transformed into a multiplication by a character in $L^2(\bR^n)^d$, i.e., 
\[
\left(R_{\cH_m} \rho_{\cH_m}(a) R_{\cH_m}^\ast f\right)(\xi,y)
= (M_{E_a I_d} f)(\xi,y)
= \enumber^{-\imunit\langle \xi,a\rangle} f(\xi)
\qquad(\xi\in\bR^n).
\]
Shortly, this means that
\[
\La_{\cF_{\al,m}}(E_{-a} I_d)
=
\rho_{\cF_{\al,m}}(a/\sqrt{\al}),
\qquad
\La_{\cH_m}
(E_{-a} I_d)
=
\rho_{\cH_m}(a).
\]
The main actors of this example are shown on Figure~\ref{fig:diagonalization_of_Weyl_translations}.
\end{example}

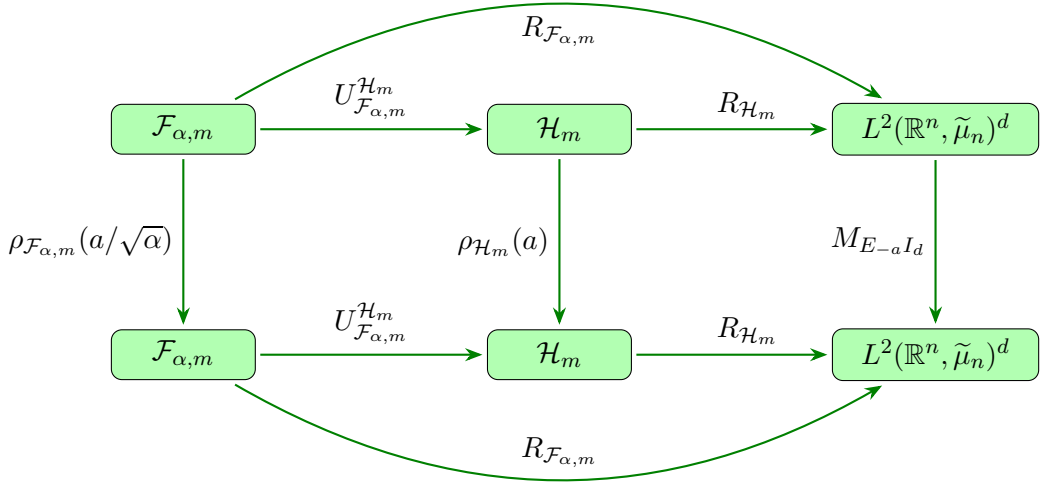
\begin{figure}[htb]
\centering
\begin{tikzpicture}
\node[spacenode] (Fm1)
  at (-5, 0) {$\cF_{\al,m}$};
\node[spacenode] (Hm1)
  at (0, 0) {$\cH_m$};
\draw [myedge] (Fm1) edge(Hm1);
\node at (-2.5, 0) [above]
  {$U_{\cF_{\al,m}}^{\cH_m}$};
\node[spacenode] (Fm2)
  at (-5, -3) {$\cF_{\al,m}$};
\node[spacenode] (Hm2)
  at (0, -3) {$\cH_m$};
\draw [myedge]
  (Fm2) edge (Hm2);
\node at (-2.5, -3) [above]
  {$U_{\cF_{\al,m}}^{\cH_m}$};
\draw [myedge]
  (Fm1) edge (Fm2);
\node at (-5, -1.5) [left]  {$\rho_{\cF_{\al,m}}(a/\sqrt{\al})$};
\draw [myedge]
  (Hm1) edge (Hm2);
\node at (0, -1.5) [left]
{$\rho_{\cH_m}(a)$};
\node[spacenode,text width=15ex] (X1)
  at (5, 0) {$L^2(\bR^n,\tmu_n)^d$};
\node[spacenode,text width=15ex] (X2)
  at (5, -3) {$L^2(\bR^n,\tmu_n)^d$};
\draw [myedge]
  (Hm1) edge (X1);
\node at (2.5, 0) [above]
  {$R_{\cH_m}$};
\draw [myedge]
  (Hm2) edge (X2);
\node at (2.5, -3) [above]
  {$R_{\cH_m}$};
\draw [myedge, bend left]
  (Fm1) edge (X1);
\node at (0, 1.3)
  {$R_{\cF_{\al,m}}$};
\draw [myedge, bend right]
  (Fm2) edge (X2);
\node at (0, -4.3)
  {$R_{\cF_{\al,m}}$};
\draw [myedge]
  (X1) edge (X2);
\node at (5, -1.5) [left]
  {$M_{E_{-a} I_d}$};
\end{tikzpicture}
\caption{Main objects of Example~\ref{example:diagonalization_of_Weyl_translations}.
\label{fig:diagonalization_of_Weyl_translations}}
\end{figure}

Given $g$ in $L^\infty(\bR^n)$,
we define $\widetilde{g}\in L^\infty(\bR^n\times\bR^n)$ by
\[
\widetilde{g}(x,y)\eqdef g(y).
\]
We denote by 
$T_{\widetilde{g}}^{\cF_{\al,m}}$
and
$T_{\widetilde{g}}^{\cH_m}$
the corresponding Toeplitz operators acting in
$\cF_{\al,m}$ and $\cH_m$, respectively.

\begin{prop}
\label{prop:gammas_of_vertical_Toeplitz_operators}
Let $g\in L^\infty(\bR^n)$.
Then
$T_{\widetilde{g}}^{\cH_m}\in\cC(\rho_{\cH_m})$,
$T_{\widetilde{g}}^{\cF_{\al,m}}\in\cC(\rho_{\cF_{\al,m}})$,
and
\[
\La_{\cF_{\al,m}}^{-1}
(T_{\widetilde{g}}^{\cF_{\al,m}})
=
\La_{\cH_m}^{-1}
(T_{\widetilde{g}}^{\cH_m})
=\ga_g,
\]
where $\ga_g:\bR^n\to\cM_d$ is the matrix-function given by
\begin{equation}\label{eq:gamma_Toeplitz}
\ga_g(\xi)_{r,s}
\eqdef
2^{n/2}\,\int_{\bR^n}
g(v)\,
\left(
\prod_{p=1}^n \psi_{\phi(r)_p}\left(\frac{\xi_p+2v_p}{\sqrt{2}}\right)
\right)
\left(
\prod_{q=1}^n \psi_{\phi(s)_q}\left(\frac{\xi_q+2v_q}{\sqrt{2}}\right)
\right)
\,
\dif\mu_n(v).
\end{equation}
\end{prop}

\begin{proof}
Indeed, according to~\cite[Proposition 8.7]{HerreraMaximenkoRamos2022},
\[
\ga_g(\xi)_{r,s}
= \int_{\bR^n} g(v) \,\conj{q_{\phi(r),\xi}(v)}\,q_{\phi(s),\xi}(v)\,\dif\tmu_n(v).
\]
Substituting~\eqref{eq:q_formula}, we get~\eqref{eq:gamma_Toeplitz}.
\end{proof}

Figure~\ref{fig:vertical_Toeplitz} illustrates the main actors of Proposition~\ref{prop:gammas_of_vertical_Toeplitz_operators}.

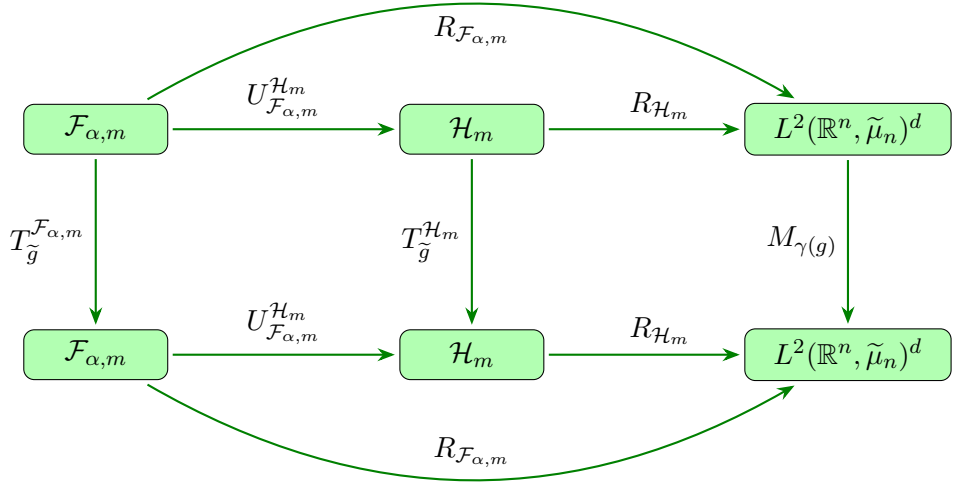
\begin{figure}[htb]
\centering
\begin{tikzpicture}
\node[spacenode] (Fm1)
  at (-5, 0) {$\cF_{\al,m}$};
\node[spacenode] (Hm1)
  at (0, 0) {$\cH_m$};
\draw [myedge] (Fm1) edge(Hm1);
\node at (-2.5, 0) [above]
  {$U_{\cF_{\al,m}}^{\cH_m}$};
\node[spacenode] (Fm2)
  at (-5, -3) {$\cF_{\al,m}$};
\node[spacenode] (Hm2)
  at (0, -3) {$\cH_m$};
\draw [myedge]
  (Fm2) edge (Hm2);
\node at (-2.5, -3) [above]
  {$U_{\cF_{\al,m}}^{\cH_m}$};
\draw [myedge]
  (Fm1) edge (Fm2);
\node at (-5, -1.5) [left]
   {$T_{\widetilde{g}}^{\cF_{\al,m}}$};
\draw [myedge]
  (Hm1) edge (Hm2);
\node at (0, -1.5) [left]
{$T_{\widetilde{g}}^{\cH_m}$};
\node[spacenode,text width=15ex] (X1)
  at (5, 0) {$L^2(\bR^n,\tmu_n)^d$};
\node[spacenode,text width=15ex] (X2)
  at (5, -3) {$L^2(\bR^n,\tmu_n)^d$};
\draw [myedge]
  (Hm1) edge (X1);
\node at (2.5, 0) [above]
  {$R_{\cH_m}$};
\draw [myedge]
  (Hm2) edge (X2);
\node at (2.5, -3) [above]
  {$R_{\cH_m}$};
\draw [myedge]
  (X1) edge (X2);
\node at (5, -1.5) [left]
  {$M_{\ga(g)}$};
\draw [myedge, bend left]
  (Fm1) edge (X1);
\node at (0, 1.3)
  {$R_{\cF_{\al,m}}$};
\draw [myedge, bend right]
  (Fm2) edge (X2);
\node at (0, -4.3)
  {$R_{\cF_{\al,m}}$};
\end{tikzpicture}
\caption{Main objects of 
Proposition~\ref{prop:gammas_of_vertical_Toeplitz_operators}.
\label{fig:vertical_Toeplitz}}
\end{figure}

\begin{rem}
Proposition~\ref{prop:gammas_of_vertical_Toeplitz_operators} is similar to~\cite[Theorem 9]{ArroyoSanchezHernandezLopez2021}.
Indeed, by taking $\eta = -\sqrt{2}\xi$ and $t = \sqrt{2}y-\eta$, 
i.e., $y = \frac{t+\eta}{\sqrt{2}}$,
we obtain the matrix-function $\si_g$
from~\cite[Theorem 9]{ArroyoSanchezHernandezLopez2021}:
\[
\si_g(\eta)_{r,s}
\eqdef
\ga_g\left(-\frac{\eta}{\sqrt{2}}\right)_{r,s}
=
\int_{\bR^n}
g\left(\frac{\eta+t}{\sqrt{2}}\right)\,
\left(
\prod_{p=1}^n \psi_{\phi(r)_p}(t_p)
\right)
\left(
\prod_{q=1}^n \psi_{\phi(s)_q}(t_q)
\right)
\,
\dif\mu_n(t).
\]
For $n=1$, this simplifies to~\cite[Theorem 4]{SanchezGonzalezLopezArroyo2018}.
\end{rem}

\begin{example}
\label{example:horizontal_convolution}
Let $h$ in $L^1(\bR^n,\tmu_n)$. Define $S_h\colon\cF_{\al,m}\to\cF_{\al,m}$ as
\[
(S_hf)(w)
\eqdef
\int_{\bR^n} \bigl(\rho_{\cF_{\al,m}}(x/\sqrt{\al})f\bigr)(w) h(x)\,\dif\tmu_n(x),
\]
i.e.,
\[
(S_hf)(w)
= \int_{\bR^n} \enumber^{\sqrt{\al}\langle w,x\rangle -\frac{1}{2}|x|^2} f\left(w-\frac{x}{\sqrt{\al}}\right) h(x)\,\dif\tmu_n(x).
\]
Notice that $S_h$ can be written as the Bochner integral of an absolutely integrable family of operators belonging to $\cC(\rho_{\cF_{\al,m}})$:
\[
S_h
\eqdef
\int_{\bR^n}
\rho_{\cF_{\al,m}}\left(\frac{x}{\sqrt{\al}}\right) h(x)\,\dif\tmu_n(x).
\]
Therefore, $S_h\in\cC(\rho_{\cF_{\al,m}})$.
Let us now compute
$Q_h\eqdef
U_{\cF_{\al,m}}^{\cH_m}
S_h U_{\cH_m}^{\cF_{\al,m}}$ and $R_{\cF_{\al,m}}S_h R_{\cF_{\al,m}}^\ast$.
\begin{align*}
(Q_h g)(u,v)
&=
(U_{\cF_{\al,m}}^{\cH_m}S_h U_{\cH_m}^{\cF_{\al,m}}g)(u,v)
= 2^{\frac{n}{2}}\enumber^{-\frac{1}{2}|u|^2-\frac{1}{2}|v|^2-\imunit\langle u,v\rangle} (S_h U_{\cH_m}^{\cF_{\al,m}}g)\left(\frac{u+\imunit v}{\sqrt{\al}}\right)
\\
&= 2^{\frac{n}{2}}\enumber^{-\frac{1}{2}|u|^2-\frac{1}{2}|v|^2-\imunit\langle u,v\rangle} 
\int_{\bR^n}\enumber^{\langle u+\imunit v,x\rangle -\frac{1}{2}|x|^2} (U_{\cH_m}^{\cF_{\al,m}}g)\left(\frac{u-x+\imunit v}{\sqrt{\al}}\right)h(x)
\,\dif\tmu_n(x)
\\
&= \enumber^{-\frac{1}{2}|u|^2-\frac{1}{2}|v|^2-\imunit\langle u,v\rangle} 
\int_{\bR^n}\enumber^{\langle u+\imunit v,x\rangle -\frac{1}{2}|x|^2} 
\enumber^{\frac{1}{2}|u-x|^2+\frac{1}{2}|v|^2+\imunit\langle u-x,v\rangle}
g(u-x, v)h(x)\,\dif\tmu_n(x)
\\
&= 
\int_{\bR^n}
g(u-x, v)h(x)\,\dif\tmu_n(x).
\end{align*}
So, $Q_h$ is just the ``horizontal convolution operator''.
By the convolution theorem,
the Fourier transform converts it to the multiplication operator:
\begin{align*}
((F\otimes I)U_{\cF_{\al,m}}^{\cH_m}S_h U_{\cH_m}^{\cF_{\al,m}}(F\otimes I)^\ast g)(\xi,v)
=
((F\otimes I)Q_h(F\otimes I)^\ast g)(\xi,v)
= g(\xi,v)\widehat{h}(\xi).
\end{align*}
Finally,
\[
(R_{\cF_{\al,m}} S_h R_{\cF_{\al,m}}^\ast f)(\xi)
=
\widehat{h}(\xi)
f(\xi)
= (M_{\widehat{h}I_d}f)(\xi).
\]
Thereby we have shown that $\La_{\cF_{\al,m}}^{-1}$ transforms $S_h$ into a scalar matrix-function:
\[
\La_{\cF_{\al,m}}^{-1}(S_h)
= \widehat{h} I_d.
\]
Notice that $\widehat{h}$ is a continuous function with zero limit at infinity.
It follows from~\cite[Theorem 26]{ArroyoSanchezHernandezLopez2021}
that the matrix-function
$\widehat{h} I_d$ belongs to the C*-algebra generated by
\[
\bigl\{\ga(g)\colon\
g\in L^\infty(\bR^n)\bigr\}.
\]
Hence, $S_h$ belongs to the C*-algebra generated by vertical Toeplitz operators.
\end{example}

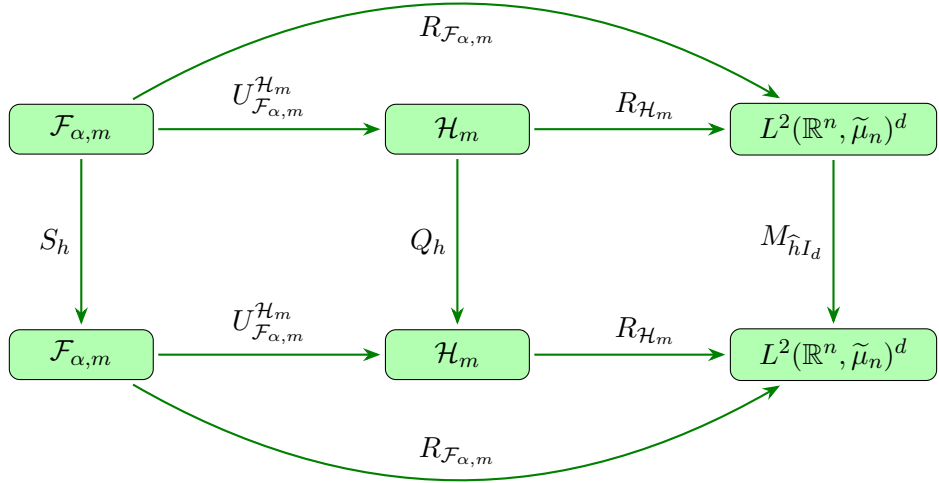
\begin{center}
\begin{tikzpicture}
\node[spacenode] (Fm1)
  at (-5, 0) {$\cF_{\al,m}$};
\node[spacenode] (Hm1)
  at (0, 0) {$\cH_m$};
\draw [myedge] (Fm1) edge(Hm1);
\node at (-2.5, 0) [above]
  {$U_{\cF_{\al,m}}^{\cH_m}$};
\node[spacenode] (Fm2)
  at (-5, -3) {$\cF_{\al,m}$};
\node[spacenode] (Hm2)
  at (0, -3) {$\cH_m$};
\draw [myedge]
  (Fm2) edge (Hm2);
\node at (-2.5, -3) [above]
  {$U_{\cF_{\al,m}}^{\cH_m}$};
\draw [myedge]
  (Fm1) edge (Fm2);
\node at (-5, -1.5) [left]
   {$S_h$};
\draw [myedge]
  (Hm1) edge (Hm2);
\node at (0, -1.5) [left]
{$Q_h$};
\node[spacenode,text width=15ex] (X1)
  at (5, 0) {$L^2(\bR^n,\tmu_n)^d$};
\node[spacenode,text width=15ex] (X2)
  at (5, -3) {$L^2(\bR^n,\tmu_n)^d$};
\draw [myedge]
  (Hm1) edge (X1);
\node at (2.5, 0) [above]
  {$R_{\cH_m}$};
\draw [myedge]
  (Hm2) edge (X2);
\node at (2.5, -3) [above]
  {$R_{\cH_m}$};
\draw [myedge]
  (X1) edge (X2);
\node at (5, -1.5) [left]
  {$M_{\widehat{h}I_d}$};
\draw [myedge, bend left]
  (Fm1) edge (X1);
\node at (0, 1.3)
  {$R_{\cF_{\al,m}}$};
\draw [myedge, bend right]
  (Fm2) edge (X2);
\node at (0, -4.3)
  {$R_{\cF_{\al,m}}$};
\end{tikzpicture}
\end{center}

\section{Numerical tests}
\label{sec:tests}

We have tested many formulas from this paper in Sagemath~\cite{Sage2023} using numerical or symbolic computations.
Notice that the symbolic integration in Sagemath is based on Maxima.

\begin{test}[decomposition of generalized Laguerre polynomials]
\eqref{eq:Laguerre_decomposition} is verified for all $n\le 8$ and $p\le 8$ with symbolic calculations in the ring of multivariate polynomials.
\end{test}

\begin{test}[reproducing kernel of $\cF_{\al,m}$ via the orthonormal basis]
The reproducing kernel of $\cF_{\al,m}$ can be expressed via the orthonormal polynomial basis:
\begin{equation}
\label{eq:K_Fm_via_basis}
K_z^{\cF_{\al,m}}(w)
=
\sum_{\substack{p,q\in\bNz^n,\\
|q|<m}}
B_{\al,p,q}(w)
\overline{B_{\al,p,q}(z)}
\qquad(w,z\in\bC^n),
\end{equation}
where $B_{\al,p,q}(w)\eqdef\prod_{j=1}^n b_{p_j,q_j}(\sqrt{\al}w_j)$,
and $b_{r,s}$ are the normalized ``complex Hermite polynomials'' given in~\cite[Section~2]{MaximenkoTelleria2020}.
The right-hand sides of~\eqref{eq:K_Fm_al_kernel}
and~\eqref{eq:K_Fm_via_basis} coincide up to $10^{-15}$ for $n,m\le 4$ and random values of $z$ and $w$ with $|z|,|w|<1/2$,
when the series in~\eqref{eq:K_Fm_via_basis} is truncated to the finite sum with $|p|\le 128$.
\end{test}

\begin{test}[reproducing property of the kernel of $\cF_{\al,m}$ for monomial functions]
Using symbolic integration we have tested the reproducing property
\[
\frac{\al^n}{\pi^n}
\int_{\bR^{2n}}
f(u+\imunit v)
\overline{
K_{x+\imunit y}^{\cF_{\al,m}}(u+\imunit v)
}
\enumber^{-\al|u+\imunit v|^2}
\dif{}u_1\cdots\dif{}u_n\,
\dif{}v_1\cdots\dif{}v_n\,
=f(x+\imunit y),
\]
for all $n,m$ with $n\le 3$, $m\le 3$,
and all $f$ of the form
$f(w)=w^p \conjw^q$
with $|q|\le m-1$ and $|p|\le 5$.
\end{test}

\begin{test}[decomposition of the reproducing kernel of $\cF_{\al,m}$ into a sum of products]
Using symbolic calculations in variables 
$\al$,
$x_1,\ldots,x_n$, $y_1,\ldots,y_n$,
$u_1,\ldots,u_n$,
$v_1,\ldots,v_n$,
we have verified that 
the right-hand sides of~\eqref{eq:K_Fm_via_sum_of_products_of_Laguerre_functions} and~\eqref{eq:K_Fm_al_kernel}
coincide for all $n,m$ with $n\le 5$ and $m\le 5$.
\end{test}

\begin{test}[Fourier connection between Laguerre and Hermite functions]
\eqref{eq:Fourier_Laguerre_inverse} and \eqref{eq:Fourier_Laguerre_function_nice} are verified with symbolic integration for all $n\le 10$.
\end{test}

\begin{test}[horizontal Fourier transform of the reproducing kernel]
We have tested
\eqref{eq:L_formula} using symbolic integration, for all $n,m$ in $\bN$ with $n,m\le 4$.
\end{test}

\bigskip\noindent
Erick Lee-Guzm\'{a}n\newline
Universidad Veracruzana\newline
Facultad de Matem\'{a}ticas\newline
Xalapa, Veracruz\newline
Mexico\newline
email: ericklee81@gmail.com\newline
https://orcid.org/0009-0001-1731-8803

\bigskip\noindent
Egor A. Maximenko\newline
Instituto Polit\'{e}cnico Nacional\newline
Escuela Superior de F\'{i}sica y Matem\'{a}ticas\newline
Ciudad de M\'{e}xico\newline
Mexico\newline
email: egormaximenko@gmail.com, emaximenko@ipn.mx\newline
https://orcid.org/0000-0002-1497-4338

\bigskip\noindent
Gerardo Ramos-Vazquez\newline
Universidad Veracruzana\newline
Facultad de Matem\'{a}ticas\newline
Xalapa, Veracruz\newline
Mexico\newline
email: ger.ramosv@gmail.com\newline
https://orcid.org/0000-0001-9363-8043

\bigskip\noindent
Armando S\'{a}nchez-Nungaray\newline
Universidad Veracruzana\newline
Facultad de Matem\'{a}ticas\newline
Xalapa, Veracruz\newline
Mexico\newline
email: armsanchez@uv.mx\newline
https://orcid.org/0000-0001-6258-8477

\end{document}